\documentclass[12pt,oneside]{amsart}
\usepackage{amsmath, amscd}
\usepackage{hyperref}
\usepackage{graphicx}
\usepackage[english]{babel}

\usepackage{amsthm}

\makeatletter
\newtheorem*{rep@theorem}{\rep@title}
\newcommand{\newreptheorem}[2]{%
\newenvironment{rep#1}[1]{%
 \def\rep@title{#2 \ref{##1}}%
 \begin{rep@theorem}}%
 {\end{rep@theorem}}}
\makeatother

      \theoremstyle{plain}
      \newtheorem{theorem}{Theorem}[section]
      \newtheorem{lemma}[theorem]{Lemma}
      \newtheorem{corollary}[theorem]{Corollary}
      \newreptheorem{theorem}{Theorem}
     \newreptheorem{lemma}{Lemma}
     
          \newtheorem{prop}[theorem]{Proposition}
     
      \theoremstyle{definition}
      \newtheorem{definition}[theorem]{Definition}
      \newtheorem{exmp}{Example}[section]
      
      \theoremstyle{remark}
      \newtheorem{remark}[theorem]{Remark}

      \newcommand{\R}{{\mathbb R}}
      \newcommand{\C}{{\mathbb C}}
      
      \newcommand{\Dc}{\text{Diff}_c}
      \newcommand{\D}{\text{Diff}}
      
        \newcommand{\s}{\mathbb{S}}

      \makeatletter
      \def\@setcopyright{}
      \def\serieslogo@{}
      \makeatother
   
\begin{document}

%



   \author[Sebastian Hurtado]{Sebastian Hurtado}
   \address{University of California, Berkeley, USA}
   \email{shurtado@math.berkeley.edu}



   

   \title[Continuity of homomorphisms of diffeomorphism groups]{Continuity of discrete homomorphisms of diffeomorphism groups}


   
  \begin{abstract}
 
  Let $M$ and $N$ be two closed $C^{\infty}$ manifolds and let $\Dc(M)$ denote the group of $C^{\infty}$ diffeomorphisms  isotopic to the identity. We prove that any (discrete) group homomorphism between $\Dc(M)$ and $\Dc(N)$ is  continuous. We also show that a non-trivial group homomorphism $\Phi: \Dc(M) \to \Dc(N)$ implies that $\dim(M) \leq \dim(N)$ and give a classification of such homomorphisms when $\dim(M) = \dim(N)$.
    \end{abstract}






   \date{\today}


   \maketitle



\begin{section}{Introduction}\label{intro}

Let $M$ be a $C^{\infty}$ manifold and let $\Dc(M)$ be the group of $C^{\infty}$ diffeomorphisms of $M$ which are isotopic to the identity and compactly supported (if $M$ is noncompact).  A well known theorem of  Mather, Herman, Thurston and Epstein states that $\Dc(M)$ as a discrete group is simple (see \cite{B}).\\

In \cite{FIL}, Filipckewicz proved the following result:  Given two $C^{\infty}$ manifolds $M$ and $N$, the groups $\Dc(M)$ and $\Dc(N)$ are isomorphic (as discrete groups) if and only if the manifolds $M$ and $N$ are $C^{\infty}$ diffeomorphic. Moreover, if such an isomorphism exists should be given by conjugation with a fixed element of $\D(M)$. Filipckewicz's theorem implies in particular that  two different smooth exotic 7-spheres have different diffeomorphism groups. \\

The focus of this paper is the study of abstract group homomorphisms between groups of diffeomorphisms. We will make no assumptions on the continuity of the homomorphisms. Let $M$ and $N$ be two manifolds and let  $\Phi: \Dc(M) \to \Dc(N)$ be a group homomorphism. Observe that the simplicity of these groups imply that $\Phi$ is either trivial or injective.\\

Some examples of homomorphisms of the type $\Phi: \Dc(M) \to \Dc(N)$ are the following: \\

\begin{enumerate}

\item \textbf{Inclusion:} If $M \subset N$ is an open submanifold of $N$, the inclusion gives a homomorphism $\Phi: \Dc(M) \to \Dc(N)$.

\item \textbf{Coverings:}  If $ p:N \to M$ is a finite covering, every diffeomorphism $f \in \D_c(M) $ can be lifted to a diffeomorphism $\tilde{f}$ of $N$. These lifts in many cases give a group homomorphism. For example, this is the case when $p$ is an irregular cover (i.e. $\text{Deck}(M/N) = \{Id\}$), as in this case he lift $\tilde{f}$ of any diffeomorphism $f$ is unique.

\item\textbf{Bundles:} The unit tangent bundle $UT(M)$, the Grassmannian $Gr_k(M)$, consisting of $k$-planes in $T(M)$, product bundles, symmetric products and other bundles over $M$ admit a natural action of $\D_c(M)$.\\

\end{enumerate}

Observe that all the homomorphisms described above  are in some sense continuous maps from $\Dc(M)$ to $\Dc(N)$ in their usual $C^{\infty}$ topologies. The main result of this paper shows that this is always the case when $M$ is closed. Before stating our main result we need the following definition:\\

\begin{definition}

For any compact subset $K \subseteq M$, let $\D_K(M)$ denote the group of diffeomorphisms in $\Dc(M)$ supported in $K$. A group homomorphism $\Phi: \Dc(M) \to \Dc(N)$ is \emph{weakly continuous} if for every compact set $K \subseteq M$, the restriction $\Phi|_{\D_K(M)}$ of $\Phi$ to $\D_K(M)$ is continuous in the weak topology, see \ref{norms}.

\end{definition}

Our main theorem is the following:\\

\begin{theorem}\label{1}

Let $M$ and $N$ be $C^{\infty}$ manifolds. If $\Phi:  {\Dc}(M) \to {\Dc}(N)$ is a discrete group homomorphism, then $\Phi$ is weakly continuous.

\end{theorem}

It is worth pointing out that there is no assumption on the dimensions of the manifolds. As a consequence of Theorem \ref{1}, we will obtain a classification of all possible homomorphisms in the case when $\dim(M) \geq \dim(N)$.\\

Mann \cite{M} gave a classification of all the non-trivial homomorphisms $\Phi:  {\Dc}(M) \to {\Dc}(N)$ in the case that $N$ is 1-dimensional. She showed using one dimensional dynamics techniques (Kopell's lemma, Szekeres theorem, etc.) that $M$ must be one dimensional and that these homomorphisms are what she described as ``topologically diagonal".\\

In the case when $M = N = \mathbb{R}$, \emph{topologically diagonal} means that $\Phi$ is given in the following way:  Take a  collection of finite disjoint open intervals $I_j \subset \R$ and diffeomorphisms $f_j : M \to I_j $.  For each diffeomorphism $g \in \Dc(\R)$ the action of $\Phi(g)$ in the interval $I_j$ is given by conjugation by $f_j$, i.e. $ \Phi(g)(x) = f_igf_i^{-1}(x)$ for every $x \in I_i$. The action is defined to be the identity everywhere else.\\

We will prove a generalization of the previous result:

\begin{theorem}\label{2}

Let $M$ and $N$ be $C^{\infty}$ manifolds, let $N$ be closed and $\dim(M) \geq \dim(N)$. If $\Phi:{\Dc}(M) \to {\Dc}(N)$ is a non-trivial homomorphism of groups, then $\dim(M) = \dim(N)$ and $\Phi$ is ``extended topologically diagonal".

\end{theorem}

The definition of extended topologically diagonal is given in Subsection \ref{etd} and is a generalization of what we previously described as topologically diagonal by allowing the possibility of taking finite coverings of $M$ and embedding them into $N$, see definition \ref{etd}.\\

Observe that as a corollary of Theorem \ref{2} we obtain another proof of Filipckewicz's theorem. The result also gives an affirmative answer to a question of Ghys (see \cite{G}) asking whether a non-trivial homomorphism $\Phi: \Dc(M) \to \Dc(N)$ implies that $\dim(N) \geq \dim(M)$.\\

The rigidity results we obtain can be compared to analogous results obtained by Aramayona-Souto \cite{SA}, who proved similar statements for homomorphisms between mapping class groups of surfaces and to  deep results of Margulis about rigidity of lattices in linear groups. (See \cite{Z}, Chapter 5).

 \begin{subsection}{Ingredients and main idea of the proof}

The key ingredient in the proof of Theorem \ref{1} is a theorem of E. Militon. This theorem says roughly that given a sequence $h_n$ converging  to the identity in $\Dc(M)$ sufficiently fast (a geometric statement), one can construct a finite subset $S$ of $\Dc(M)$, such that the group generated by $S$ contains the sequence $h_n$ and such that the word length $l_S(h_n)$ (See \ref{product}) of each diffeomoprhism $h_n$ in the alphabet $S$ is bounded by a sequence $k_n$ not depending on $h_n$ (an algebraic statement), see \ref{mil} for the correct statement. Militon's theorem is a generalization of a result of Avila \cite{A} about $\D(S^1)$ and is related to results of Calegari-Freedman \cite{CF}.\\ 

 Militon's result is strongly related to the theme of distortion elements in geometry group theory (See Gromov \cite{Gro}, Chapter 3). An element $f$ of a finitely generated group $G$ with generating set $S$ is called distorted if $\underset{n \to \infty}{\lim} \frac{l_S(f^n)}{n} = 0$. \\
 
 Consider for example the group $BS(2,1)= \{a,b | bab^{-1} = a^2 \}$. This group can be thought as the group generated by the functions $a: x \to x+1$ and $b: x \to 2x$ in $\D(\R)$. Observe that $b^{n}ab^{-n} = a^{2^{n}}$ for every $n$ and therefore $a$ is distorted (In some sense $a$ is ``exponentially" distorted). \\

An element $f \in \Dc(M)$  is said to be recurrent in $\Dc(M)$ if $f$ satisfies that $\lim\inf d_{C^{\infty}}(f^{n}, \text{Id}) = 0$, for example an irrational rotation in $S^1$ is recurrent (See \ref{norms}). A corollary of Militon's theorem is that any recurrent element in $\Dc(M)$ is  distorted in some finitely generated subgroup of $\Dc(M)$. Even more, it implies that such an element is``arbitrarily distorted" in the sense that $f$ could be made as distorted as one wants (See \ref{CF} for a precise definition). For a nice exposition of the concept of distortion in transformation groups and the definition of the different types of distortion, see \cite{CF}. \\

Observe that distorted elements are preserved under group homomorphisms and therefore if one wants to understand the existence of a group homomorphism, it might be fruitful to understand the distortion elements of the groups involved in the homomorphism. One example of this approach can be seen in the work of Franks and Handel \cite{FH}. They proved that any homomorphism from a large class of higher rank lattices $\Gamma$ (For example $SL_3(\mathbb{Z})$) into ${\D}_{\mu}(S)$ (the group of diffeomorphisms of a surface preserving a measure $\mu$) is finite using the fact that $\Gamma$ has distorted elements and showing that any distorted element $f \in \D_{\mu}(S)$ fixes the support of the measure $\mu$. For a more concrete statement of this and a clear exposition, see \cite{F}.\\

The main idea of how to use Militon's theorem to obtain our rigidity results is illustrated in section \ref{motivation} with a motivating example that encloses the main idea of Theorem \ref{1} and is the heart of this paper. 

\end{subsection}

\begin{subsection}{Outline}

The paper is divided as follows. In section \ref{main}, we will use Militon's theorem to prove Lemma \ref{m2}, which is the main tool to prove Theorem \ref{1}. In section \ref{conti}, we prove Theorem \ref{2} assuming that $\Phi$ is weakly continuous, this section is independent of the other sections. In section \ref{blowups}, we establish some general facts about two constructions in group actions on manifolds that we will use for the proof of Theorem \ref{1}. In section \ref{proof}, we show that $\Phi$ is weakly continuous using the results on section \ref{main} and the facts discussed in section \ref{blowups}, finishing the proof of theorems \ref{1} and \ref{2}. Finally, in section \ref{remarks}, we end up with some questions and remarks related to this work.

\end{subsection}

\begin{subsection}{Acknowledgements}

The author would like to thank the referee, Kathryn Mann, Benson Farb and especially Emmanuel Militon for useful comments and corrections on a previous version of this paper. The author is in debt with his advisor, Ian Agol, for his help, encouragement and support.

\end{subsection}

\end{section}

\begin{section}{Main technique}\label{main}

This is the most important section of the chapter, we will prove Lemma \ref{m2} which is a slightly weaker version of Theorem \ref{1}. The main ingredient in the proof of Lemma \ref{m2} is Theorem \ref{mil}. We will discuss a motivating example in subsection \ref{motivation} which avoids some necessary technicalities in the proof of Lemma \ref{m2} but contains the central point of our discussion.

\begin{subsection}{Notation}\label{defweakt}

Throughout this section and for the rest of this chapter we make use of the following notation:  $M$ and $N$ denote $C^{\infty}$ manifolds, $\Dc(M)$ will be the group of  $C^{\infty}$ diffeomorphisms which are compactly supported and isotopic to the identity. The letter $\Phi$ will always be a discrete group homomorphism $\Phi: \Dc(M) \to \Dc(N)$ unless stated otherwise. For a compact set $K$, we let $\D_K(M)$ be the subgroup of $\Dc(M)$ consisting of diffeomorphisms supported in a compact set $K$.\\

We will consider $\Dc(M)$ as a topological space endowed with the weak topology, this topology is metrizable and we will denote by $d$ any metric that induces such topology. For more information about the weak topology see \ref{norms}.\\

\end{subsection}

We are now in position of stating the main lemma of this section.

\begin{lemma}\label{m2}

Let $M$ and $N$ be $C^{\infty}$ manifolds, let $K$ be any compact subset of $M$ and let $\Phi:  {\Dc}(M) \to {\Dc}(N)$ be a discrete group homomorphism. Suppose $h_n$ is a sequence in ${\D_K}(M)$ such that $$ \lim_{n \to \infty} d_{C^{\infty}}(h_n, \text{Id}) = 0$$ Then $\{\Phi(h_n)\}$ contains a  subsequence converging to a diffeomorphism $H$, which is an isometry for a $C^{\infty}$ Riemannian metric on $N$.

\end{lemma}

Observe that the isometry $H$ in Lemma \ref{m2} is necessarily homotopic to the identity, and therefore if the manifold $N$ does not admit a metric with isometries homotopic to the identity (as for example is the case when $N$ is a closed surface of genus $g \geq 2$), the previous theorem implies the continuity of the homomorphism $\Phi$ restricted to $\D_K(M)$.\\

Before stating Militon's theorem, we will need the following standard definition:\\

\begin{definition}\label{product}

Let G be a finitely generated group and let $S$ be a finite generating set. We assume that $S= S^{-1}$. Given an element $g \in G$, the word length $l_{S}(g)$ of $g$ in $s$ is defined as the minimum integer $n$ such that $g$ can be expressed as a product $$g = s_1s_2...s_n$$ where each $s_i \in S$. In a more technical language, $l_{S}(g)$ is the distance of $g$ to the identity in the Cayley graph corresponding to $S$.

\end{definition}

We are now in position of stating Militon's theorem:

\begin{theorem}\label{mil}

Let $M$ be a compact manifold. There exist two sequences $\{\epsilon_n\} \to 0$ and $\{k_n\}$ of positive real numbers  such that for any sequence $h_n$ of diffeomorphisms in ${\Dc}(M)$ satisfying: $$ d_{C^{\infty}}(h_n, \text{Id}) \le \epsilon_n$$ there exists a finite set $S \subset  {\Dc}(M)$ such that:

\begin{enumerate}

\item $h_n$ belongs to the subgroup generated by $S$
\item $l_{S}(h_n) \leq k_n$

\end{enumerate}

\end{theorem}

\begin{remark}

The previous theorem is also true in the case where $M$ is not compact and $\Dc(M)$ is replaced with $\D_K(M)$ for a compact subset $K$ of $M$.

\end{remark}

\begin{remark}

The set $S$ depends on the sequence $h_n$, but the sequences $\{\epsilon_n\}$ and $\{k_n\}$ are independent of the choice of $h_n$.

\end{remark}

This result is a generalization of a Theorem of Avila \cite{A}, who proved the same result for  $\Dc(\s^1)$ case. The proof of the theorem is related to a construction of Calegari-Freedman \cite{CF}, who showed that an irrational rotation of $\s^2$ is ``arbitrarily distorted" in ${\Dc}(\s^2)$.\\

For sake of completeness, we give a rough summary of how Theorem \ref{mil} is proved in \cite{A}, \cite{Mil}. The proof of Theorem \ref{mil} has two steps:  The first step is to show that an element in $\Dc(M)$ close to the identity can be written as a product of commutators of elements  which are close to the identity,  this step is  proved using a KAM theory technique. The first step allows one to reduce the general case to the case where all the $f_n$'s are commutators supported in  small open balls. The second step consist of using a finite set of diffeomorphisms  to encode the sequence $\{f_n\}$ into a finite $S$, this step uses some clever algebraic tricks similar to the ones used in the proof of the simplicity of $\Dc(M)$ by Thurston (see \cite{B}, p. 23). The tricks depend heavily on the fact that the $f_n$'s are products of commutators.\\

To give a hint of how powerful Theorem \ref{mil} can be for our purposes, observe the following: Let $h_n$ be a sequence in $\Dc(M)$ supported in a compact set $K$ and converging to the identity in $\D_c(M)$ as in Theorem \ref{mil}.  As a consequence of Theorem \ref{mil},  the sequence $\{h_n\}$ is generated by a finite set of diffeomorphisms $S$ and therefore the set $\Phi(S)$ generates a group containing $\{\Phi(h_n)\}$. As a consequence, there exists a compact set $K' \subset N$ (Namely, the union of the support of the generators) such that every element of the sequence  $\{\Phi(h_n)\}$ is supported in $K'$. This statement is not easy to deduce by elementary means.\\

The following example is fundamental and contains the main idea in the proof of Lemma \ref{m2}:

\begin{subsection}{Motivating example}\label{motivation}

Let $\s^2$ be the 2-sphere and let $\Phi: \Dc(\s^2) \to \Dc(\s^2)$ be a homomorphism of groups. As a consequence of Theorem \ref{2} we are going to show that $\Phi$ is induced by conjugation via an element of $\D(\s^2)$. In particular, $f$ and $\Phi(f)$ turned out being conjugates as diffeomorphisms for every $f \in \Dc(\s^2)$. In the next proposition we show that this is true for the most basic diffeomorphisms of $\s^2$, namely rotations in $SO(3)$.\\

\begin{prop}

Let $\s^2$ be the 2-sphere and let $\Phi: \Dc(\s^2) \to \Dc(\s^2)$ be a group homomorphism. For any rotation $R$ in $SO(3)$, we have that $C := \Phi(R)$  preserves a Lipschitz metric $g$. If the metric $g$ happens to be $C^{\infty}$ then $C$ is conjugate to a rotation in $\Dc(\s^2)$.   

\end{prop}

\begin{proof}(Sketch) First consider the case when $R$ is a finite order rotation.  In this case, the diffeomorphism $C$ is an element of finite order, let say of order $n$. If we take any Riemannian metric $g'$ in $\s^2$ and we consider the average metric $g = \sum_{i=0}^{n-1} (C^{i})^{*}g'$, it is easy to check that $g$ is a Riemannian metric that is invariant under $C$ . It follows from the uniformization theorem of Riemann surfaces that $g$ is conformally equivalent to the standard metric in $\s^2$. This implies that $C$ is conjugate to a conformal map of finite order of $\s^2$ (an element of $PSL_2(\C)$) and this easily implies that $C$ is conjugate to a finite order rotation in SO(3).\\

Now consider the case when $R$ has infinite order (an irrational rotation). We want to imitate the proof of the finite order case:  Let $g'$ be an arbitrary Riemannian metric and consider the sequence of  average metrics  $g_n = \frac{1}{n}\sum_{i=0}^{n} ({C}^i)^{*}g'$. If an appropriate subsequence of $\{g_n\}$ converges to a metric $g$, the metric $g$ would be invariant under $C$ and we would be able to conclude that $C$ is conjugate to a rotation arguing in the same way as in the finite order case. The problem is that to obtain such a convergent subsequence, we need the sequence $\{(C^{n})^{*}(g')\}$ to not behave badly.\\

The idea then is to obtain bounds in the derivatives of powers of $C$ which would assure that the sequence of average metrics $g_n$ have a nice convergent subsequence. Therefore, we are going to show that $\|D(C^n)\| \leq K$, for some fixed constant $K$ and every integer $n$.\\ 

Suppose such a constant $K$ does not exist, then there is a subsequence $\{C^{m_n}\}$ such that $\|D(C^{m_n})\| \to \infty$. Passing through a subsequence, we can assume that the sequence of rotations $R^{m_n}$ is convergent,  let's assume it converges to the identity (if $R^{m_n}$ converges to another rotation, a similar argument works). Furthermore, we can assume that $\|D(C^{m_n})\| \geq e^{e^{k_n}}$ and that $d_{C^{\infty}}(R^{m_n}, \text{Id}) \leq \epsilon_n $, for the constants $\epsilon_n$ and $k_n$ in Theorem \ref{mil}.\\

Using Theorem \ref{mil}, we can find a finite set $S = \{s_1,s_2  \dots s_l\}$ such that  $l_{S}(R^{m_n}) \leq k_n$ for every $n$ and therefore we can express $R^{m_n} = \prod_{j=1}^{l_n} s_{i_j}$  for some  integers $l_n \leq k_n$. If we take a constant $L$ such that for each generator $\|D(\Phi(s_m))\| \leq L$ for every $m \leq l$, we obtain: $$ g(n) = \|DC^{m_n}\| =  \| D \big( \prod_{j=1}^{l_n}\Phi(s_{i_j}) \big)\|  \leq \prod_{j=1}^{l_n}\|D(\Phi(s_{i_j}))\|    \leq L^{k_n}$$
 
 This give us a contradiction: On the one hand  the sequence $\|D(C^{m_n})\|$ diverges superexponentially (greater than $e^{e^{k_n}}$), and on the other hand, it diverges at most exponentially (less than $L^{k_n}$).\\
 
 It is not difficult to show this imply  $g_n$ has a convergent subsequence converging to a metric that a priory might be not smooth, see the proof of \ref{m2} for more details.\\

\end{proof}

To promote the metric $g$ in the previous proposition to an actual $C^{\infty}$ metric we need to have control over higher derivatives of diffeomorphisms. We would also need an analog for higher order derivatives of the the fact that for any $f, g \in \D(\s^2)$, the inequality $ \|D(f\circ g)\| \leq \|D(f)\|\|D(g)\|$ holds. The purpose of the next two subsections is  to set up the right framework to achieve these tasks.

\end{subsection}


\begin{subsection}{Norms of derivatives}\label{norms}

In the next two sections we will define ``norms" $\|.\|_r$ of any element in $\Dc(M)$ whose purpose is to measure how big the $r$-derivatives of a diffeomorphism are. We are going to use the word ``norm'' to refer to them but the functions $\|.\|_r$ are not in any sense related to the various definitions of norms in the literature. This and the next subsection are technical and might be a good idea to skip them in a first read. If the reader decides to do that, he/she should have in mind that the  $r$-norm $\|.\|_r$ of a diffeomorphism is an analog of $\|D(.)\|$ for the the $r$-derivative and to take a look at Lemma \ref{ineqnorms}, where using the chain rule, we will  obtain for any $f,g \in \Dc(M)$ a bound for $\|f\circ g\|_r $ in terms of the  $r$-norms of the diffeomorphisms $f$ and $g$.\\

The group $\D(M)$ has a natural topology known as the ``weak topology". This topology is defined as follows: Take a locally finite covering by coordinate charts $\{(U_i, \phi_i)\}_i$ of $M$.\\

For every diffeomorphism $f \in \Dc(M)$, every compact set $K \subset U_i$ such that $f(K) \subset U_j$, every integer $r \geq 0$ and every real number $\epsilon > 0$, there is a neighborhood $\mathcal N_{K,\epsilon}^{i,j}(f)$ of $f$ defined as the set of all diffeomorphisms $g \in \D(M)$ such that $g(K) \subset U_j$ and such that for every $0 \leq k \leq r$: $$  \|D^k(\phi_j \circ f \circ \phi_i^{-1}) - D^k(\phi_j \circ g \circ \phi_i^{-1})\| \leq \epsilon$$

Where for any function $f : \R^n \to \R^n$, $D^k(f)$ denotes any $k$-partial derivative of a component of $f$.\\

The topology induced in $\Dc(M)$ as a subset of $\D(M)$ happens to be metrizable, $\Dc(M)$ is actually a Frech\'et manifold (see \cite{Hir} for more details about this topology). We will denote by ``$d_{C^{\infty}}$" any metric which induces such topology in $\Dc(M)$.\\ 

For the proof of Lemma \ref{m2} we will need our norm $\|.\|_r$ to have two properties. The first property is that the $r$-norm should be able to tell when the $r$-derivatives of  a sequence of diffeomorphisms $\{f_n\}$ is diverging or not. The second property is to have a bound for the $r$-norm of $f \circ g$ in terms of the $r$-norms  of $f$ and $g$.  Based on these needs and the definition of the topology of $\Dc(M)$,  we will define for an integer  $r \geq 1$  the $r$-norms in $\Dc(M)$ as follows:\\
 
Let's start first with the case when $M \subseteq \R^n$ is an open subset. In this case, we define for any compactly supported diffeomorphism $f \in \Dc(M)$ and for any $r \geq 1$, $$\| f\|_r = \max \left| \frac{\partial^r f}{\partial x_{i_1}\partial x_{i_2}..\partial x_{i_r}} \right|$$  The maximum is taken over all the partial derivatives of the different components of $f$ and over all the points $x \in M$. \\

In order to define $\| . \|_r$ for an arbitrary manifold,  we use a covering set $\{ U_i \}_i$ of $M$ by coordinate charts satisfying the following properties:\\

\begin{enumerate}
 
\item Each $U_i$ is diffeomorphic to a closed ball of $\R^n$ by  some diffeomorphism $\phi_i$.
\item $(\text{int}(U_i), \phi_i)$ is a system of coordinate charts for $M$.
\item Every compact set $K$ intersects a finite number of $U_i$'s.
\end{enumerate}

The existence of such covering follows easily from the paracompacity of $M$.\\

Let fix a compact set $K$ in $M$. After choosing an arbitrary covering with the properties above, we define for  every $f \in \D_K(M)$: $$\|f\|_{r,K} = \sup_{i,j} \| \phi_j \circ f \circ \phi_i^{-1}\|_r$$ The supremum is taken over all  $(i,j)$'s such that $U_i\cap K \neq \emptyset$, $U_j\cap K \neq \emptyset$, and such the expression $\phi_j \circ f \circ \phi_i^{-1}$ is well defined. The set of $(i,j)$'s satisfying these properties is finite and therefore the above expression is always finite.\\

\begin{remark}\label{normsremark} Whenever the compact set $K$ is clear in our context we will suppress the subindex $K$ and denote $\|.\|_{r,K}$ by $\|.\|_{r}$, but it is important to take into account that $\|.\|_r$ is depending on $K$.
\end{remark}

\begin{lemma}\label{aa}

Let $f_n$ be a sequence in ${\D_K}(M)$. Suppose there exist constants $C$ and $C_r$ (for every $r \geq 1$) such that the following holds:

\begin{enumerate}
\item  $\|D(f_n)\| \leq C$
\item $\|f_n\|_r \leq C_r$ for $n$ sufficiently large (depending on $r$)
\end{enumerate}

Then there is a subsequence of $\{f_n\}$ converging in the $C^{\infty}$ topology to a diffeomorphism $f$ satisfying $\|D(f)\| \leq C$ and $\|f\|_r \leq C_r$.

\end{lemma}

\begin{proof}

From the bound  $\|D(f_n)\| \leq C$, we conclude that  the sequence $f_n$ is equicontinuous, and by Arzela-Ascoli's theorem we obtain a subsequence $f'_n$ of $f_n$ converging uniformly to a continuous map $f: K \to M$.\\

Similarly, for a fixed integer $r \geq 1$,  the inequality $\|f'_n\|_r \leq C_r$ implies there is a subsequence $f_n''$ of $f'_n$ such that the first $(r-1)$ derivatives of $f_n''$ converge uniformly in any fixed coordinate chart intersecting $K$. Using Cantor's diagonal argument, we obtain a subsequence $f_n'''$ of $f''_n$, such that all the $r$-derivatives of $f'''_n$ converge uniformly, for every coordinate charts involving $K$. This implies that the map $f$ is $C^{\infty}$.\\

 
 Using that $\|D(f_n)\| \leq C$, we can obtain that $f$ is bi-Lipschitz and therefore invertible. In conclusion $f$ is a diffeomorphism and satisfies $\|D(f)\| \leq C$.\\

\end{proof}

Observe that even though our definition of the $\|.\|_r$ norms depends strongly on the choice of coordinates charts, the previous statement is independent of the particular choice.

\begin{subsubsection}{\textbf{Metrics}}

In the proof of Lemma \ref{m2}, we will also need to be able to measure when a sequence of metrics $g_n$ is ``diverging" in the $C^{\infty}$ sense:\\

For a metric $g = \sum_{i,j} g_{i,j}  dx_idx_j$ defined in an closed sets $U \subseteq \R^n$, we define: 
\[
\|g\|_r = 
\begin{cases}
\max_{i,j, x \in U}{|g_{i,j}(x)|} & \text{for } r = 0\\
\max_{i,j}\|g_{i,j}\|_r    & \text{for } r \geq 1 
\end{cases}
\] 


We defined the previous norm $\|.\|_r$ just for metrics defined in closed subsets of $\R^n$, the reason being that this shorten the use of new terminology and the length of our proofs.\\

In the same spirit of Lemma \ref{aa}, if a sequence of metrics $g_n$ defined in a closed set $U \subset \R^n$ satisfy that  $\|g_n\|_r \leq D_r$ for some fixed constants $D_r$, it is easy to show hat $g_n$ has a subsequence converging to a $C^{\infty}$ metric $g$ uniformly in each compact set. We will need to make use of the following fact, which is an easy consequence of the chain rule:

\begin{lemma}\label{metric}

Let $U$, $V$ be closed sets of $\R^n$. Let $g$ be any Riemannian metric defined in $U$ and  let $H: V  \to U$ be a diffeomorphism.  Suppose for every $r\geq1$, there exist constants  $C_r$, $D_r$ such that:

\begin{enumerate}
\item $\|g\|_r \leq D_r$ 
\item $\|H\|_r \leq C_r$

\end{enumerate}

Then, there are constants $D'_r$, just depending on $C_k$ and $D_k$ for $k \leq r+1$, such that $\|H^{\ast}(g)\|_r \leq D'_r$ for every $r$.

\end{lemma}

\begin{remark} $H^{\ast}(g)$ denotes the pullback of $g$ under $H$.
\end{remark}

\end{subsubsection}

\begin{subsubsection}{\textbf{Main inequality}}

The next lemma is the analog for higher derivatives of the inequality $\|D(f\circ g)\| \leq \|D(f)\| \|D(g)\|$ that we announced previously.\\

\begin{lemma}\label{ineqnorms}
For every integer $r \geq 2$, there exist constants $C_r$ such that for any $f, g \in {\Dc}(M)$ the following inequality holds: $$\| f \circ g\|_r \leq C_{r}(\max_{1 \leq l \leq r}  \{ \|f\|_l + \|g\|_l \})^{r+1}$$

\end{lemma}

\begin{proof}

We will first consider the case when $M = \R^n$. Applying the chain rule $r$ times to each of the components of the partial derivatives of $f \circ g$ we obtain a sum of terms.  Each term in this sum is the product of at most $r+1$ partial derivatives of some $f_i$'s and $g_j$'s ($f_i$, $g_j$ denote components of $f$ and $g$). The number of terms in this sum is at most a constant $C_r$ (just depending on $r$ and the dimension of $M$) and each of the terms is less than $\|f\|_l$ + $\|g\|_l$ for some $1 \leq l \leq r$, therefore:
$$\| f \circ g\|_r \leq C_r(\max_{1 \leq l \leq r}  \{ \|f\|_l + \|g\|_l\})^{r+1}$$

For an arbitrary manifold $M$, we take a covering by coordinate charts $(U_i, \phi_i)$ as in the definition of the $r$-norm. For every pair $(i,j)$ such that $((f\circ g)^{-1}(U_j)) \cap U_i \neq \emptyset$ and every coordinate chart $(U_k, \phi_k)$  such that $f^{-1}(U_k)$ intersects the set $((f\circ g)^{-1}(U_j)) \cap U_i$, we have that:

\begin{align*}
\| \phi_j \circ f\circ g \circ \phi_i^{-1}\|_r &=  \| \phi_j \circ f \circ \phi_k^{-1} \circ \phi_k \circ g \circ \phi_i^{-1}\|_r  \\
&\leq C_r(\max_{1 \leq l \leq r}  \{ \|\phi_j \circ f \circ \phi_k^{-1}\|_l + \|\phi_k \circ g \circ \phi_i^{-1}\|_l\})^{r+1} \\
&\leq C_r(\max_{1 \leq l \leq r}  \{ \|f\|_l + \|g\|_l\})^{r+1}
\end{align*}

In the previous equation, the first inequality  follows from the $\R^n$ case proved previously. We can conclude the following: $$\| f \circ g\|_r =  \sup_{i,j} \| \phi_j \circ f \circ g \circ \phi_i^{-1}\|_r  \leq C_{r}(\max_{1 \leq l \leq r}  \{ \|f\|_l + \|g\|_l\})^{r+1}$$

\end{proof}

\end{subsubsection}

\end{subsection}


\begin{subsection}{Proof of Lemma \ref{m2}}

The following lemma is the only place where we use Militon's theorem. Roughly, the lemma claims that there exists a compact set $\mathcal{K} \subset \Dc(N)$ such that for any sequence $h_n$ converging to the identity and $n$ sufficiently large, $\Phi(h_n)$ is as close to $\mathcal{K}$ as one wants. More precisely:

\begin{lemma} \label{m1}

Let $M$ and $N$ be two manifolds, let  $K$ be any compact subset of $N$, and  let $\Phi:  {\Dc}(M) \to {\Dc}(N)$ be a group homomorphism. There exists a compact set $K'$ and constants $C$ and $C_r$ such that: For every sequence $h_n$ in $\D_K(M)$ such that  $$ \lim_{n \to \infty} d_{C^{\infty}}(h_n, \text{Id}) = 0 $$ The following holds for $n$ sufficiently large:

\begin{enumerate}

\item $\Phi(h_n)$ is supported in $K'$.
\item $\| D(\Phi(h_n))\| \leq C$.
\item $\|\Phi(h_n)\|_r \leq C_r$.

\end{enumerate}

\end{lemma}

\begin{remark}
How large $n$ has to be in order to  $\|\Phi(h_n)\|_r \leq C_r$ to hold might depend on each particular $r$. The $r$-norms $\|.\|_r$ in $N$ are defined with respect to the compact set $K'$ in (1).
\end{remark}

In the proof of Lemma \ref{m1} we will use the following technical fact, which is a consequence of Lemma \ref{ineqnorms}.

\begin{lemma}\label{ineq}

 Fix an integer $r \geq 2$.  Let $C_i$ be the constants defined in \ref{m1} and  $s_1, s_2  \dots s_k$ be elements of ${\Dc}(M)$. Let $L$ be a constant such that: 
\begin{enumerate}

\item $\|s_j\|_i \leq L$,  for every $1\leq i \leq r$ and $1\leq j \leq k$.
\item $\underset{1\leq i \leq r}\max\{C_i, 2\} \leq L$

\end{enumerate}

for every $k \geq 1$ the following inequality holds:  $$\|s_1\circ s_2 \dots \circ s_k\|_r \leq L^{(r+2)^{2k}} $$ 

\end{lemma}

\begin{proof}

The proof goes by induction on $k$, the case $k=1$ is obvious. Suppose it holds for an integer $k$. Using Lemma \ref{ineqnorms}, we have that:

\begin{align*}
 \|s_1\circ s_2 \dots \circ s_{k+1}\|_r &\leq  L(\max_{1 \leq i \leq r} \{ \|s_1\circ s_2  \dots s_k\|_i + \|s_{k+1}\|_i \})^{r+1}  \\
&\leq L(\max_{1 \leq i \leq r} \{ \|s_1\circ s_2  \dots s_k\|_i +L\})^{r+1}   \\
&\leq L(L^{(r+2)^{2k}} + L)^{r+1} \\
&\leq L(2L^{(r+2)^{2k}})^{r+1} \\
&\leq L(L^{r+1}L^{(r+2)^{2k+1}})\\
&= L^{(r+2)^{2k+1} + (r + 2)}\\
\end{align*}


Using that the inequality $(r+2)^{2k+1} + (r + 2) \leq (r+2)^{2(k+1)}$ holds for $r\geq1$, the results follows for $k+1$.

\end{proof}

\begin{proof}{(Lemma \ref{m1}):}\\

\textbf{Existence of $K'$}: The proof goes by contradiction. Suppose such a compact set $K' \subset N$ does not exist, using Cantor's diagonal argument one can finds a sequence $h_n \to \text{Id}$ such that the supports of $\Phi(h_n)$ are not contained in any compact set. Passing to a subsequence we can assume that $ d_{C^{\infty}}(h_n, \text{Id}) \leq \epsilon_n$ for the sequence $\{\epsilon_n\}$ in Theorem \ref{mil}. Using theorem \ref{mil} we obtain a finite set $S = \{s_1,s_2  \dots s_l\}$ such that the sequence $\{h_n\}$ is contained in the group generated by $S$. Applying $\Phi$ the same is true for the sequence $\{\Phi(h_n)\}$ and $S' = \{\Phi(s_1),\Phi(s_2)  \dots \Phi(s_l)\}$, therefore the support of all the $\Phi(h_n)$'s lies in the compact set $K''$ given by the union of the supports of the elements in $S'$, which gives us a contradiction.\\

\textbf{Bound for $ \| D(.)\|$}: Suppose that no such constant $C$ exists. Using a diagonal argument we can find a sequence $h_n \in \D_K(M)$  converging to the identity and such that the sequence $g(n) := \|D(\Phi(h_n))\| $ diverges. Furthermore, passing to a subsequence we can assume that $ d_{C^{\infty}}(h_n, \text{Id}) \le \epsilon_n$ and  that $\underset{n \to \infty}\lim  \frac{\log(g(n))}{k_n}= \infty$, for the constants  $\epsilon_n$ and $k_n$ defined in Theorem \ref{mil}.\\

 Using Theorem \ref{mil}  we can find a finite set $S = \{s_1,s_2  \dots s_l\}$ in $\Dc(M)$ generating a group containing al the $h_n$'s and such that  $l_{S}(h_n) \leq k_n$ for every $n$. Therefore, we have that $h_n = \prod_{j=1}^{l_n} s_{i_j}$  for some constants $l_n \leq k_n$. If we take a constant $L$ such that $\|D(\Phi(s_i))\| \leq L$ for every generator $s_i$  we get that: $$ g(n) = \|D(\Phi(h_n))\| = \big{\|} \prod_{j=1}^{l_n}D(\Phi(s_{i_j})) \big{\|} \leq \prod_{j=1}^{l_n}\|D(\Phi(s_{i_j}))\|    \leq L^{l_n} \leq L^{k_n}$$

Therefore the inequality $\frac{\log(g(n))}{k_n} \leq \log(L)$ holds, contradicting the fact that $ \frac{\log(g(n))}{k_n}$ diverges.\\

\textbf{Bound for $ \| .\|_r$}:  The proof is similar to the previous case, but we will need to use Lemma \ref{ineq}. Suppose that for a fixed integer $r$ such constant $C_r$ does not exist. In that case, we can find a sequence $h_n$ converging to the identity and such that $g(n) := \|\Phi(h_n)\|_r $ diverges. Passing to a subsequence, we can further assume that $ d_{C^{\infty}}(h_n, \text{Id}) \le \epsilon_n$ and also that: 
\begin{align}
\underset{n \to \infty}\lim  \frac{\log \log(g(n))}{k_n}= \infty
\end{align} 
For the sequences $(\epsilon_n)$ and $(k_n)$ in Lemma \ref{mil}.\\

 Using Lemma \ref{mil}, we can find a finite set $S$ generating a subgroup of $\D_c(M)$ where the inequality  $l_{S}(h_n) \leq k_n$ holds for every $n$. Let $L$ be a constant such that $\|\Phi(g)\|_r \leq L$ for every $g \in S$ and which satisfies the hypothesis of Lemma \ref{ineq}. Using Lemma \ref{ineq}, we have that:  $$ g(n) = \| \Phi(h_n) \|_r \leq L^{(r+2)^{2k_n}}$$ Taking logarithms in both sides we obtain: $$ \log \log(g(n)) \leq 2k_n\log(r+2) + \log(\log(L)) $$ which contradicts (1).
 
\end{proof}

Finally, we will formulate Lemma \ref{m2} one more time and finish its proof.\\

\begin{replemma}{m2}

Let $M$ and $N$ be $C^{\infty}$ manifolds and let $\Phi:  {\Dc}(M) \to {\Dc}(N)$ be a group homomorphism. Suppose that $h_n$ is a sequence in ${\Dc}(M)$ supported in a compact set $K$ and such that $$ \lim_{n \to \infty} d_{C^{\infty}}(h_n, \text{Id}) = 0$$ Then $\{\Phi(h_n)\}$ has a convergent subsequence, converging to a diffeomorphism $H$, which is an isometry for a $C^{\infty}$ Riemannian metric on $N.$

\end{replemma}

\begin{proof}

From Theorem \ref{m1}  we conclude that given a fixed integer $r\geq 1$, for $n$ large enough the inequality $\|\Phi(h_n)\|_r \leq C_r$ holds . Using Lemma \ref{aa} we can extract  a convergent subsequence from $\{\Phi(h_n)\}$ which is converging to a $C^{\infty}$ diffeomorphism $H$. Even more, taking powers of $h_n$  we have the following: For every fixed integer $k$,  the sequence $h_n^k \to \text{Id}$ and then  the inequality $\|\Phi(h_n^k)\|_r \leq C_r$ holds for $n$ large enough. Therefore we can conclude that $\| H^k\|_r \leq C_r$ for every integer $k$.\\
 
Fix a Riemannian metric $g'$ on $N$ and consider the sequence of metrics: $$g_n = \frac{1}{n+1}\sum_{k=0}^{n} (H^k)^{*}g'$$

(The notation $(H^k)^{*}g'$ denotes the pullback of the metric $g'$ under the diffeomorphism $H^{k}$). Next, we will show that the sequence $g_n$ has a subsequence converging to a $C^{\infty}$ metric $g$, which is invariant under $H$.  \\

To simplify the rest of the proof we will use the following notation. Take coordinate charts $(U_i, \phi_i)$ of $N$ as in the definition of the $\|.\|_r$ norms (See \ref{norms}). For any metric $g$ in $N$, let $g^{i}$ be the metric defined in the closed set $\phi_i(U_i) \subset \R^n$ obtained by pull back, more concretely $g^i := (\phi_i^{-1})^{*}g$. \\

Going back to our proof, observe that a diagonal argument shows that $g_n$ has a convergent subsequence if and only if in any coordinate chart, the corresponding sequence of metrics has a convergent subsequence. More concretely, we just need to show that for any fixed integer $i$, the sequence of metrics $\{g^{i}_n\}_n$ (as defined in the previous paragraph) have a convergent subsequence.\\

 Observe that  $g_n^i = \frac{1}{n+1}\sum_{k=0}^{n} ((H^k)^{*}g')^i$. From (1) on Theorem \ref{m1} the diffeomorphism $H$ is supported in a compact set, therefore for each fixed $i$ the set $\bigcup_k H^k(U_i)$ is contained in the union of a finite number of $U_j$'s and so we can assume it intersects non-trivially the sets $U_l$ for $ 1 \leq l \leq j$.  We define $D_r := \underset{1\leq l \leq j} \max \|g^l\|_r$.\\

For every integer $k$ and each $x \in U_i$, the point $H^{k}(x)$ lies in the interior of $U_l$ for some $1 \leq l \leq j$. Thus we have that $((H^k)^{*}g)^i =(\phi_l \circ H^{k} \circ \phi_i^{-1})^{*} g^l $ in a small ball around $x$. Using the inequalities $\|\phi_l \circ H^{k} \circ \phi_i^{-1}\| \leq C_r$ and $\|g^l\| \leq D_r$,  we can apply Lemma \ref{metric} to obtain constants $D'_r$ such that  $0 \leq \|((H^k)^{*}g)^i\|_r \leq D'_r$ for every integer $r\geq 1$. The previous inequality implies that $\|g_n^i\|_r \leq D'_r$ for every $n\geq1$ and therefore our sequence $\{g_n^i\}_n$ has a convergent subsequence as we wanted.\\

In conclusion, we can extract from the sequence $g_n$ a subsequence converging to a $C^{\infty}$ metric $g$. To finish our proof we just need to show that  the metric $g$ is non-degenerate  and invariant under $H$.\\

\textbf{Non degeneracy}: As the inequality $\| D(H^k)\|  \leq C$ holds for every integer $k$, we conclude that for each vector $v \in TM$, each of the metrics $g_{n}$ satisfy the inequality: $$ \frac{g'(v,v)}{C^2} \leq g_{n}(v,v)  \leq C^2g'(v,v)$$ This implies that any convergent subsequence of $\{g_n\}$ converge to a non-degenerate metric and so our $g$ is non-degenerate.\\

\textbf{Invariance under $H$}: For any vector $v \in TN$, we have:  

\begin{center}

$\|H^{\ast}g_n(v,v) - g_n(v,v) \| = \frac{\|(H^{n+1})^{*}g(v,v) - g(v,v)\|}{n+1} \leq (C+1)\frac{\|g(v,v)\|}{n+1}$

\end{center}

Taking limits for our subsequence in the previous inequality we obtain that the metric $g$ is invariant under $H$. 

\end{proof}

\end{subsection}

As a corollary, we obtain a proof of Theorem \ref{1} in the case $N$ is non compact.

\begin{corollary}

Let $M$ and $N$ be $C^{\infty}$ manifolds and suppose that $N$ is non compact,  any group homomorphism $\Phi: \D_c(M) \to \Dc(N)$ is weakly continuous.

\end{corollary}
 
 \begin{proof}
 
 Let $K$ be any compact set in $M$. Arguing by contradiction, suppose  $h_n$ is a sequence converging to the identity in $\D_K(M)$ but such that $\Phi(h_n)$ is not converging to the identity. We can suppose that $d_{C^{\infty}}(\Phi(h_n), \text{Id}) \geq C >1$ for every $n \geq 1$. By Lemma \ref{m2} there is subsequence of $\Phi(h_n)$ converging to an isometry $H$ in $N$ and by (1) in Theorem \ref{m1} the isometry $H$ should be the identity outside a compact set $K'$ of $N$, therefore $H$ is the identity everywhere contradicting that  $d_{C^{\infty}}(\Phi(h_n), \text{Id}) \geq C >1$.

 \end{proof}

\end{section}

\begin{section}{Continuous case}\label{conti}

\begin{subsection}{Notation}

Let $M$ and $N$ denote two $C^{\infty}$ manifolds of dimension $m$ and $n$ respectively. For the rest of this section we will assume that $N$ is closed. If $U$ is an open subset of $M$, we define: $$G_U = \{ f \in \Dc(M) | \text{ supp}(f) \subset U \}$$

We recall the definition of `` weak continuity" given in section \ref{intro}.

\begin{definition}

For any compact subset $K \subseteq M$, let $\D_K(M)$ denote the group of diffeomorphisms in $\Dc(M)$ supported in $K$. A group homomorphism $\Phi: \Dc(M) \to \Dc(N)$ is \emph{weakly continuous} if for every compact set $K \subseteq M$, the restriction $\Phi|_{\D_K(M)}$ of $\Phi$ to $\D_K(M)$ is continuous.

\end{definition}

\begin{remark}If $M$ is closed, this definition is equivalent to $\Phi$ being continuous in the usual sense.
\end{remark}

\end{subsection}

\begin{subsection}{Weak Continuity vs Continuity}

The concept of weak continuity for a group homomorphism $\Phi: \Dc(M) \to \Dc(N)$ is not equivalent to our homomorphism $\Phi$ being continuous in the weak topology of $\Dc(M)$ and one should not  expect such continuity to hold. For example, take the open unit ball $B \subset \R^n$. The homomorphism induced by the inclusion $i: B \to \R^n $ is a group homomorphism which is not a continuous map in the weak topology:\\

Take a sequence of diffeomorphisms $f_n$ supported in small disjoint balls $B_n$ contained in $B$ such that each $B_n$ is centered at  a point $p_n$ and has radius $r_n$ in such a way that the sequence $p_n$ converges to a point in $\partial \bar{B}$ , the radii $r_n \to 0$  and such that $D_{p_n}(f_n) = 2\text{Id}$.  The sequence $f_n$ converges to the identity in $\Dc(B)$ in the weak topology because restricted to every compact set $K \subset B$, the sequence $f_n$ is  the identity for large $n$. Nonetheless $f_n$ does not converge in $\Dc(\R^n)$ and therefore $\Phi$ is not continuous.\\
 
 \end{subsection}

In this section, we will assume that $\Phi$ is always \emph{weakly continuous}. The main purpose of this section is to establish the following result:\\

\begin{theorem}\label{cont}

If $\Phi: {\Dc}(M) \to {\Dc}(N) $ is a non-trivial weakly continuous homomorphism of groups, then $\dim(M) \leq \dim(N)$. If $\dim(M) = \dim(N)$, then $\Phi$ is extended topologically diagonal.

\end{theorem}

The definition of ``extended topologically diagonal'' will be given in \ref{etd} . The assumption of weak continuity results to be very strong for a homomorphism between groups of diffeomorphisms. As a proof of that, we recall the following classic result (See  \cite{MZ}, Ch. 5):

\begin{theorem}\label{mz}

(Montgomery, Zippin) Let $\Phi$ be a continuous homomorphism from a finite dimensional Lie group $G$ to the group ${\Dc}(M)$ of $C^{\infty}$ diffeomorphisms  of a manifold $M$. Then the map $\pi: M \times G \to M$ given by $\pi(x,g) = \Phi(g)x$ is a $C^{\infty}$ map in both variables $(x,g)$ simultaneously.\\

\end{theorem}

\begin{corollary}\label{flow}

 If $\Phi: {\Dc}(M) \to {\Dc}(N) $ is weakly continuous and $X$ is a compactly supported vector field in $M$ generating a flow $f_t \in \Dc(M)$, then $\Phi(f_t)$ is a flow in $\Dc(N)$ generated by some vector field $Y$ supported in $N$.

\end{corollary}

\begin{proof}

$H(t) = \Phi(f_t)$ defines a homomorphism from $\R$ to $\D(M)$. $H(t)$ is continuous by the weak continuity of $\Phi$ and therefore is a $C^{\infty}$ flow by Theorem \ref{mz}. We obtain that $\Phi(f_t)$ is generated by some vector field $Y$ in $N$.
 
\end{proof}

In our proofs, we will use the following known fact. For a proof of this fact, see \cite{B}.

\begin{lemma}{(Fragmentation property)} Let $M$ be a manifold and $\mathcal{B}$ be an open cover of $M$, then $\Dc(M)$ is generated by the following set: $$ \mathcal{C} = \{f \in \Dc(M) \text{ such that }  supp(f) \subset B_0\text{, for some } B_0 \in \mathcal{B}\}$$ 

\end{lemma}

In the next two subsections we will construct some examples of homomorphisms between groups of diffeomorphisms when $\text{dim}(M) = \text{dim}(N)$.\\

 \begin{subsection}{Examples}
 
When $\dim(M)= \dim(N)$, there are two fundamental ways of constructing examples of homomorphisms of the type $\Phi: \Dc(M) \to \Dc(N)$. The first one is to embed a finite number of copies of $M$ into $N$ and to act in such copies in the obvious way (topologically diagonal). The second one consist of homomorphisms coming from lifting diffeomorphisms of $M$ to a covering $N$. The main theorem of this section shows that if $\Phi$ is continuous, all  possible homomorphisms are combinations of these two examples. We will discuss these two fundamental examples in more detail:
 
 \end{subsection}
 
 \begin{subsubsection}{Topologically diagonal}\label{td}
 
 \begin{definition}\label{td1}

A homomorphism $\Phi: \Dc(M) \to \Dc(N)$ is  \emph{topologically diagonal} if there exists a finite collection of disjoint open sets $U_i \subseteq N$ and diffeomorphisms $\rho_i : M \to U_i$ , so that for every $f \in {\Dc}(M)$ we have:
\[
\Phi(f)(x) = 
  \begin{cases}
  \rho_i \circ f  \circ \rho_i^{-1}(x)  & \text{ if } x \in U_i\\
   x                                           & \text{     otherwise }\\
  \end{cases}
\]

\end{definition}

The fact that $U_i$ is a finite collection is necessary. Even though we may be able to embed an infinite collection of $M$'s in a closed manifold $N$, if the collection is not finite, the homomorphism $\Phi$ is not well defined.To illustrate better this point, observe the following:\\

Let $M = \R^n$ and suppose there is an infinite collection of open sets $U_i \subset N$ giving rise to a homomorphism $\Phi$ as in \ref{td1}. Take a diffeomorphism $f$ supported in the unit  ball around the origin in $\mathbb{R}^n$,  such that $f(0) = 0$ and $Df(0) = 2\text{Id}$.\\

Using the notation of $\rho_i$ as in  \ref{td1}, consider the sequence $\rho_i(0) = p_i$ in $N$. $\{p_i\}_i$ is an infinite sequence by assumption and therefore has an accumulation point $p$ in $N  \setminus \cup_i{U_i}$.  We have that $D(\Phi(f))(p_i) = 2\text{Id}$, and by continuity we have that $D(\Phi(f))(p) = 2\text{Id}$. It's not difficult to see by the compactness of $N$ that there is a another set of points $q_i \to p$  such that $\Phi(f)q_i= q_i$ and that $D(\Phi(f))(q_i) = \text{Id}$, which implies that $D(\Phi(f))(p) = \text{Id}$, a contradiction.\\
 
  \end{subsubsection}

 \begin{subsubsection}{Finite coverings}\label{covers}
 
 Many times when $N$ is a finite covering of $M$ there is a homomorphism $\Phi: \Dc(M) \to \Dc(N)$ obtained by lifting a given $f \in \Dc(M)$ to an appropriate diffeomorphism $\tilde f$ of $N$. For example, if $M$ and $N$ are closed hyperbolic surfaces, every diffeomorphism of $M$ isotopic to the identity can be lifted to the universal covering $\mathbb{H}^2$ (Hyperbolic plane). Even more, it can be lifted to a unique homeomorphism $f':  \mathbb{H}^2 \to  \mathbb{H}^2$ such that $f'$ restricts to the identity in $\partial \mathbb{H}^2 = \s^1$. This lift commutes with all the covering translations, and so one can project $f'$ to a diffeomorphism $\tilde f$ of our covering surface $N$. One can check that  $\widetilde{fg} = \tilde{f} \tilde{g}$ by the uniqueness of these lifts. Therefore, we obtain a homomorphism $\Phi: \Dc(M) \to \Dc(N)$.\\
 
Unfortunately, this is not true for all covering maps $\pi: M \to N$. For example, if we let  $M = N = \s^1$ and $\pi: \s^1 \overset{2}\to \s^1$, the obvious $2$ cover, any of the lifts of the rotation $r_2$ in $\s^1$ of order 2, has order 4, and therefore it is impossible to construct a homomorphism $\Phi$ that commutes with $p$. The obstruction as we will see, comes from the fact that $\pi_1(\Dc(\s^1)) \cong \mathbb{Z}$ is generated by a full rotation around the circle and that $\pi: \s^1 \overset{2}\to \s^1$ corresponds to the subgroup $2\mathbb{Z}$ of $\mathbb{Z}$.\\
 
 In a more precise way, we have the following situation: Let $G = \Dc(M)$, $H$ be the group consisting of lifts of elements of $G$ to $N$ and $L$ be the group of deck transformations corresponding to the covering map $\pi: M \to N$, we have an exact sequence of discrete groups:

 \begin{equation}\label{es}
 e \to  L \to H \to G \to e
 \end{equation}
 
We are interested in understanding when this sequence splits.\\

For sake of completeness, we mention a useful criterion for proving this. For a proof, see \cite{Br} p.104.

\begin{lemma} The sequence $e \to  L \to H \to G \to e$ splits if and only if the two following conditions are satisfied:

\begin{itemize}

\item The natural map $H \to Out(L)$ lifts to $Aut(L)$.
\item The element in $H^3(G, Z(L))$ defined by the exact sequence is trivial ($Z(L)$ denotes the centralizer of $L$). 
 
 \end{itemize}
\end{lemma}  

As a consequence of Theorem \ref{1} (whose proof will be given in section \ref{proof}) our  homomorphism $\Phi: G \to H$ is  continuous (when $M$ and $N$ are closed) and therefore showing that the previous sequence splits discretely is equivalent to show it splits continuously. Therefore, we just need to understand in which case there exists a continuous homomorphism $\Phi: G \to H$ which splits our sequence. \\

Observe that $H$ is a covering of $G$, and so the existence of such continuous section $\Phi$ is equivalent to the covering $H$ being trivial: $H$ as a topological space should consist of  the union of disjoint copies of $G$. This happens if and only if every element of $L$ lies in different components of $H$. In conclusion, our sequence splits if there is no deck covering transformation $L$ isotopic to the identity in $H$.\\

In conclusion, our exact sequence \eqref{es} splits if and only if there is no deck transformation $L$ of the covering $\pi: N \to M$, which is isotopic to the identity in $\Dc(N)$. \\

Observe that if there is a path $\gamma \in H$ connecting the identity to a non-trivial deck transformation $L$, projecting $\gamma$ down to $G$ we obtain a loop $\gamma'$ representing a non-trivial element in $\pi_1(G, e)$, the fundamental group of $G$ based at the identity.\\ 

To obtain a more precise statement of when this happens, we will need the following definition:\\

Take an arbitrary point $x \in M$, there is a natural evaluation map $E_x: \Dc(M) \to M$ given by $E_x(f) = f(x)$ and so we have a natural homomorphism of fundamental groups: $${E_x}^{*}: \pi_1(\Dc(M), e) \to \pi_1(M,x)$$

Observe the following fact:

\begin{prop} A closed loop $\gamma \in \pi_1(G,e)$ has a lift to $H$ connecting the identity to a non-trivial deck transformation if and only if ${E_x}^{*}(\gamma) \not \in \pi_1(N)$.

\end{prop}

\begin{proof}
In one direction, if a loop  $\gamma \in \pi_1(G,e)$ satisfies ${E_x}^{*}(\gamma) \not\in \pi_1(N)$, the lift of ${E_x}^{*}(\gamma)$ to $N$ is not a closed loop. Therefore,  the lift of $\gamma$ to $H \subset \Dc(N)$ starting at the identity has to have the other endpoint in a non-trivial lift of the identity of $\Dc(M)$, that is, a non-trivial deck transformation. The other direction is trivial.\\ 

\end{proof}

Summarizing all the previous discussion, we have the following criterion:

 \begin{lemma} \label{crit}
 Let $\pi:  N \to M$ be a covering map. There exists a unique continuous homomorphism ${\pi}^*: \Dc(M) \to \D_c(N)$  such that for every $f \in \Dc(N)$, the following diagram commutes:
 
 \[
 \begin{CD}
N @>{{\pi}^{*}(f)}>> N \\
 @V{\pi}VV          @VV{\pi}V \\
M @>>{f}> M
 \end{CD}
\]

if and only if $ E_x^{*}(\pi_1(\D_c(M), e)) \subset \pi_1(N)$.

\end{lemma}

Observe that if $f_t$ is a path in $\Dc(M)$ representing a loop $\alpha$ of $\pi_1(\Dc(M), e)$ based at the identity and $\gamma_s$ is a loop based at $x$ representing an element of $\pi_1(M)$, the family of paths $\{f_t(\gamma_s)\}$ give us a homotopy between the loops $\gamma$ and $(E_x^{*}(f_t))^{-1}\gamma(E_x^{*}(f_t))$. This shows that the image of ${E_x}^{*}$ is a central subgroup of $\pi_1(M)$. In particular, we obtain the following useful corollary:

\begin{corollary} If $\pi: N \to M$ is a covering map and $\pi_1(M)$ is center-less, then there is a homomorphism $\pi^{*}: \Dc(M) \to \Dc(N)$ as in Lemma \ref{crit}.

\end{corollary}

\end{subsubsection}

\begin{subsection}{Extended topologically diagonal}

Based on the previous two examples we make the following definition:

\begin{definition} \label{etd}

A homomorphism $\Phi: \Dc(M) \to \Dc(N)$ is  \emph{extended topologically diagonal} (See Figure \ref{fig}) if there exists a finite collection of  disjoint open sets $U_i \subseteq N$ and  a set of  finite coverings $\pi_i: M_i \to M$, together with a collection of diffeomorphisms $\rho_i: M_i \to U_i$ such that:\\

For every $f \in {\Dc}(M)$, the following holds:

\[
\Phi(f)(x) = 
  \begin{cases}
  \rho_i \circ \pi_i^{*}(f)  \circ \rho_i^{-1}(x)  & \text{ if } x \in U_i\\
   x                                           & \text{     otherwise }\\
  \end{cases}
\]
 
Where $\pi_i^{*}: \Dc(M) \to \Dc(M_i)$ are the homomorphisms coming from the finite coverings $\pi_i$ as in the previous subsection.

\end{definition}

\begin{remark} The coverings have to be finite in order for $\Phi$ to be an actual homomorphism into $\Dc(N)$. See the comments at the end of \ref{td}.

\end{remark}

We will illustrate the previous definition with the following example:

\begin{figure}\label{fig}
  \centering
    \includegraphics[width=0.8\textwidth]{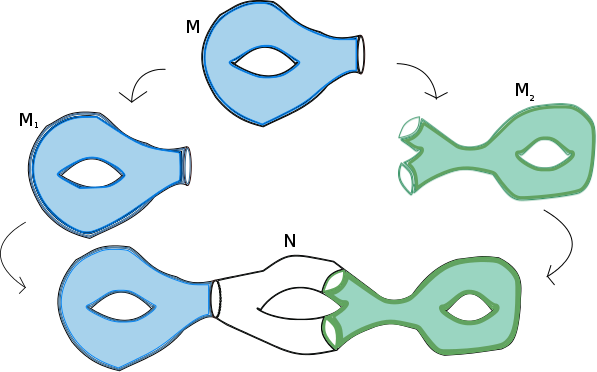}
    \caption{An extended topologically diagonal embedding}

\end{figure}
\begin{exmp}Let $M$ be the once punctured 2-dimensional torus and let $N$ be the closed orientable surface of genus $3$. Take $M_1$ to be a copy of $M$ and $M_2$ to be the twice punctured torus (which is a double cover of $M$) as in  Figure \ref{fig}. We embed $M_1$ and $M_2$ into $N$ as illustrated in Figure \ref{fig} and obtain a homomorphism $\Phi: \Dc(M) \to \Dc(N)$.\\
\end{exmp}
We are now in position to give a proof of the main result of this section.

\end{subsection}

\begin{subsection}{Proof of Theorem \ref{cont}}
In order to prove Theorem \ref{cont}, we will need to make the following definitions:\\ 

\begin{definition}


For any group $G \subseteq {\Dc}(N)$, we define: $$supp_1(G) = \{p \in N \text{ such that } g(p) \neq p \text{ for some }  g \in G\}$$  
\end{definition}

For each $x \in M$, let's define: $$S_x = \bigcap_{r > 0} supp_1(\Phi(G_{B_r}))$$ where $B_r$ is  an open ball of radius $r$ around $x$. We construct the sets $S_x$ thinking of constructing  the maps $\rho_i$ in the definition of \emph{ extended topologically diagonal}. We would like to show that $S_x$ consists of a discrete set of points and that the correspondence $S_x \to x $ give us a collection of coverings from some open sets of $N$ to $M$. \\

Let $N' := N \backslash \text{fix}(\Phi({\Dc}(M)))$  be the set of points in $N$ that are not fixed by the action of $\Dc(M)$.\\

\begin{lemma}\label{prop}
Let $\Phi: {\Dc}(M) \to {\Dc}(N)$ be a non-trivial  homomorphism,  for every $x \in M$ we have the following:

\begin{enumerate}

\item $S_x$ is a non-empty set.
\item For any $f \in {\Dc}(M)$, $\Phi(f)S_x = S_{f(x)}$.
\item $N' = \cup_{x \in M}  S_x$.

\end{enumerate}
\end{lemma}

\begin{proof}

 We start with the proof of item (3). If $p \in N'$, there is an element $f_0$ in ${\Dc}(M)$ such that $\Phi(f_0)(p) \neq p$. Using the fragmentation property, we can express $f$ as a product of diffeomorphisms supported in balls of radius one an so we obtain a diffeomorphism $f_1$ supported in $B_1 \subset M$ such that $\Phi(f_1)(p) \neq p$. We can iterate this inductive procedure and construct a diffeomorphism $f_n \in {\Dc}(M)$ supported in a ball $B_n$ of radius at most $1/n$, such that $B_n \subset B_{n-1}$ and such that $\Phi(f_n)p \neq p$. We define $x := \bigcap_n B_{n}$. It is clear that $p \in S_x$, so item (3) follows.\\
 
Item (2) follows by conjugating $f$ and applying $\Phi$. Item (1) is a direct consequence of (2) and (3). 

\end{proof}

The following two lemmas are technical lemmas needed for the proof of Lemma \ref{inj2}.\\

\begin{lemma}\label{rn}

Given $x_0 \in M$ and $U \subseteq M$ an open set containing $x_0$, there exists a neighborhood ${\mathcal{N}}_{x_0}$ of $x_0$ in $M$  and a continuous map $\Psi:\mathcal{N}_{x_0} \to G_U$ such that for every $x \in {\mathcal{N}}_{x_0}$ we have that $\Psi(x)(x_0) = x$.

\end{lemma}

\begin{proof}

Using coordinate charts is enough to check the case  when $M = \R^n$ and $x_0 = 0$. For $1\leq i \leq n$, let $X_i$ be vector fields in $\R^n$ supported in $U$ such that $\{X_i(0)\}_i$ is a linearly independent set of vectors. If $f^i_t$ are the flows generated by $X_i$, we have that $f^i_t \in G_U $.\\

We define for $t = (t_1,t_2 \dots ,t_n) \in (-\epsilon,\epsilon)^n$, the map $\rho: (-\epsilon,\epsilon)^n \to \R^n$  by  $\rho(t) = f^1_{t_1}f^2_{t_2} \dots f^n_{t_n} (0)$. By our choice of $X_i$'s, for sufficiently small $\epsilon$, $\rho$ is  a diffeomorphism onto a small neighborhood $\mathcal{N}_0$ of $0 \in \R^n$.\\

If $\rho^{-1}(x) = (t_1(x),t_2(x), \dots ,t_n(x))$,  we can define  the map $\Psi:\mathcal{N}_{0} \to G_U$ by $\Psi(x) = f^1_{t_1(x)}f^2_{t_2(x)} \dots f^n_{t_n(x)}$. $\Psi$ is clearly continuous and $\Psi(x)(0) = x$.

\end{proof}

 \begin{lemma}\label{inj1}
 Let $\Phi: {\Dc}(M) \to {\D}(N)$ be a weakly continuous homomorphism and let $p$ be a point in $N$. There are no $n+1$ pairwise disjoint open subsets $U_i \subset M$, $i =1,2, \dots ,{n+1}$, such that every subgroup $\Phi(G_{U_i})$ does not fix $p$.

\end{lemma}

\begin{proof}

Suppose that such a point $p$ exists. As a consequence of the simplicity of $G_{U_i}$, any element of $G_{U_i}$ can be written as a product of a finite number of flows. Therefore, for every $1 \leq i \leq n+1$, there exists a flow $f^i_t$ supported in $U_i$, such that $\Phi(f^i_t)$ does not fix $p$. Observe that the vector fields $X_i$ generating the flows $\Phi(f^i_t)$ satisfy that $X_i(p) \neq 0$. Also,  as $\Phi(f^i_t)$ and  $\Phi(f^j_t)$ commute, we have that $[X_i,X_j] = 0$ for every $i,j$.\\

Let's take a maximal collection of vector fields $\mathcal{C} = \{X_1, X_2, \dots , X_k\}$ in $N$  such that $X_1(p), X_2(p), \dots ,X_k(p)$ are linearly independent and such that all the $X_j$'s comes from flows in different $\Phi(G_{U_i})$'s. Observe that $k$ must be less or equal to $n$, therefore we can assume w.l.o.g. that if $f_t$ is any flow in $G_{U_{n+1}}$ and $Y$ is the vector field corresponding to the flow $\Phi(f_t)$, then $Y(p)$ is a linear combination of elements in $\{X_1(p), X_2(p), \dots ,X_k(p)\}$.\\

If we consider the open set: $$U= \{ q \in N \mid X_1(q), X_2(q), \dots , X_k(q) \text{ are linearly independent}\}$$ We have that the vector fields  in $\mathcal{C}$ define a foliation $F$ on $U$  by Frobenius integrability theorem. Let $L$ be the leaf  of $F$ containing $p$. Observe that as $Y(p)$ is a linear combination of elements of $\{X_1(p), X_2(p), \dots ,X_k(p)\}$ and $[Y, X_i] = 0$, then we have that $Y$ is tangent everywhere to the leaves defined by $\mathcal{C}$. It also follows  that if $q \in U$, then $\Phi(g_t) q \in U$ for every $t$. In conclusion, we have that  $\Phi(g_t)$ necessarily preserves $L$.\\

As the diffeomorphism $g_t$ was an arbitrary flow in $G_{U_{n+1}}$, we have obtained a non-trivial homomorphism $\Psi: G_{U_{n+1}} \to {\D}(L)$. Even more, the image of this map is abelian, since if $Y$ and $Z$ are vector fields coming from $\Phi(G_{U_{n+1}})$, both $Y$ and $Z$ commute with all the elements of $\mathcal{C}$, and therefore $Y$ and $Z$ can be shown to be linear combination of  vector fields in $\mathcal{C}$. This contradicts the simplicity of $G_{U_{n+1}}$.\\

\end{proof}

\begin{lemma}\label{inj2} If $x,y$ are 2 distinct points on $M$, then $S_x\cap S_y = \emptyset $.

\end{lemma}
 
\begin{proof}

From Lemma \ref{inj1},  for any $n+1$ different points $x_1, x_2 \dots ,x_{n+1}$ in $M$ we have that $S_{x_1} \cap S_{x_2}\cap  \dots \cap S_{x_{n+1}} = \emptyset$. Let $k$  be the largest integer such that there exist  $x_1, x_2, \dots ,x_k$ distinct points on $M$ such that $S_{x_1} \cap S_{x_2}\cap  \dots \cap S_{x_{k}} \neq \emptyset$. We are going to show that $k=1$.\\

Take $p$ be an arbitrary point  in $S_{x_1} \cap S_{x_2}\cap  \dots \cap S_{x_{k}}$. By Lemma \ref{rn}, we can find pairwise disjoint neighborhoods $\mathcal{N}_{x_i}$ of $M$ and continuous maps $\Psi_i: \mathcal{N}_{x_i} \to \Dc(M)$, such that $\Psi_i(y)(x_i) = y$ for every $y \in \mathcal{N}_{x_i}$. Even more, we can suppose that the support of the elements in the image of $\Psi_i$ and $\Psi_j$ have disjoint support if $i\neq j$.\\

Let $V = \underset{i}\prod\mathcal{N}_{x_i} \subseteq M^k$, we define the map $\Psi: V  \to N$ given by: $$\Psi(y_1,y_2,  \dots ,y_k) =\Phi( \underset{i}\prod \Psi_i(y_i))(p)$$

$\Psi$ is continuous because $\Phi$ and all the $\Psi_i$'s are continuous. Observe that  $(\underset{i}\prod \Psi_i(y_i)) (x_l) = y_l $ for every $ 1\leq l\leq k$, therefore $\Phi( \underset{i}\prod \Psi_i(y_i))(p) \in  S_{y_1} \cap S_{y_2}\cap  \dots \cap S_{y_{k}}$ by property $(2)$ in Lemma \ref{prop}. By maximality of $k$, we obtain that   
$\Psi$ is injective. It follows then that $\dim(M^k) \leq \dim(N)$ and therefore $k=1$ and $\dim(M)= \dim(N)$.

\end{proof}

As a corollary of the last paragraph we obtain:

\begin{corollary}
If $\Phi: {\Dc}(M) \to {\D}_c(N)$ is a weakly continuous homomorphism and $\dim(N) < \dim(M)$, then $\Phi$ is trivial.
\end{corollary}

 \begin{corollary} If $\Phi: {\Dc}(M) \to {\Dc}(N)$ is a weakly continuous homomorphism of groups and $\dim(M) = \dim(N)$, then for each $p \in N'$ we have that $\Phi(\Dc(M))p$ is an open set of $M$.
  
 \end{corollary}
 
Using the previous corollary, we obtain that for our set $N':= N \setminus \text{fix}(\Phi(\Dc(M)))$ defined before Lemma \ref{prop}, we have that $N' = \cup_iU_i$ where each  $U_i$ is an open set defined as $U_i := \Phi(\D(M))p_i$ for some point $p_i \in N'$.  Observe that each  $U_i$ is a connected open set where the action of $\Dc(M)$ is transitive. To finish the proof of Theorem \ref{cont}, we just need to show the following: 
   
 \begin{lemma} 
 
 If $\Phi(\D_c(M))$ acts transitively on $N$, then there is a  $C^{\infty}$ covering map $\pi: N \to M$  such that for every $f \in \D_c(M)$ the following diagram commutes:
 \[
 \begin{CD}
N @>{\Phi(f)}>> N \\
 @V{\pi}VV          @VV{\pi}V \\
M @>>{f}> M
 \end{CD}
\]
 
 In other words, $\Phi = {\pi}^{*}$ as in subsection \ref{covers}.
 
 \end{lemma}

 \begin{proof}
 
 Let $\pi: N \to M$ be the map defined by $\pi(p) = x$ if $p \in S_x$. Observe that the map is well defined by Lemma \ref{inj2} and is surjective by our transitivity assumption. Next, we will  show that $\pi$ is locally a $C^{\infty}$ diffeomorphism.\\ 
 
 Using Lemma \ref{rn}, we can find a neighborhood $\mathcal{N}_{x}$ of $M$ and a continuous map $\Psi: \mathcal{N}_{x} \to \Dc(M)$ such that $\Psi(y)(x) = y$ for every $y \in \mathcal{N}_{x}$. We can define the map $\rho: \mathcal{N}_{x}  \to N$ given by $$\rho(y) =\Phi(\Psi(y))(p)$$
 
 Observe that $\rho$ is continuous because $\Phi$ is weakly continuous and injective because $\rho(y) \in S_y$ and Lemma \ref{inj2}. We also have that $ \pi\circ\rho= \text{Id}$ which implies that $\pi$ is a covering map. To show that $\pi$ is a $C^{\infty}$ map it is enough to show that  $\rho$ is  a $C^{\infty}$ diffeomorphism onto its image.\\
 
 The fact that $\rho$ is smooth follows from Lemma \ref{mz}: It's enough to prove that the image of a $C^{\infty}$ embedded curve $\gamma$ passing through $x$ is  a $C^{\infty}$ curve. Given $\gamma$,  we can construct a $C^{\infty}$ compactly supported vector field $X$ such that $ \gamma$ is a flow line of this vector field. $X$ defines a flow $f_t$ and by Lemma \ref{flow} we have that $\Phi(f_t)$ is a $C^{\infty}$ flow. Observe that $\rho(f_t(x)) \in S_{f_t(x)} = \Phi(f_t)S_x$. Using that $\rho$ is locally a homeomorphism for $t$ sufficiently small we have that $\rho(f_t(x)) =  \Phi(f_t)p$ and therefore $\rho(\gamma)$ is smooth. 
 
 \end{proof}
 
 \end{subsection}

\end{section}


\begin{section}{Blow-ups and gluing of actions}\label{blowups}

In the proof of Theorem \ref{1} we are going to make use of  the following two constructions on group actions on manifolds:

\begin{subsection}{Blow-ups}\label{blowups1}

Given  a closed manifold $N$ and a closed embedded submanifold $L \subset N$ of positive codimension, let $\Dc(N,L)$ be the group of diffeomorphisms of $N$ preserving  $L$ setwise. Suppose we have a group $G \subseteq \Dc(N,L)$. We will show how to construct a natural action of $G$ in $N^{\sigma}$, where $N^{\sigma}$ is the manifold obtained by blowing up $N$ along $L$.\\
 
Intuitively, the blow-up $N^{\sigma}$ is obtained by compactifying $N \backslash L$ by the unit normal bundle of $L$ in $N$. More concretely,  let fix a Riemannian metric on $N$ and let $N(L)$ be the unit normal bundle of our submanifold $L$ in that metric. The normal neighborhood theorem states that the map: $$p_0: N(L)\times(0,\epsilon) \to N$$ Defined by sending $(v,t)$ to the point $\gamma_v(t)$ in $N$ at distance $t$ along the geodesic with initial vector $v$, is a diffeomorphism onto a neighborhood of $L$ for  sufficiently small $\epsilon$.\\

We now define formally $N^\sigma$,  the  \emph{blow-up} of $N$ along $L$, as follows:  $$N^\sigma = (N(L)\times[0,\epsilon)) \cup (N\backslash L) / p_0.$$

 $N^{\sigma}$ is a compact manifold with $N(L)$ as its boundary and there is a canonical smooth map $\pi: N^{\sigma} \to N$. The smooth structure on the blow-up $N^{\sigma}$ does not depend on the metric, $N^{\sigma}$ can be also described as the compactification of $N \backslash L$ by the set of directions up to positive scale in $ \underset{q \in L}\cup (T_qN / T_qL)$.\\
  
In order to describe the smooth structure on $N^{\sigma}$ in more detail, suppose for the rest of this section that $L$ consists of a single point $p$ (the more general case can be treated in a similar way).\\ 

Let $x_1,x_2, \dots , x_n$ be coordinate charts around $p$ in $N$, where $x_i(p) = 0$ for every $1\leq i \leq n$ and let $v$ be a direction up to positive scale in $T_p(N)$. Without loss of generality, assume that in our coordinate charts we have that $v = \langle v_1,v_2, \dots ,v_{n-1},1\rangle$. Then, there is a corresponding set of coordinate charts $\xi_1,\xi_2, \dots ,\xi_{n-1}, x_n$  around $v$ in $N^{\sigma}$, where $v$ corresponds to the point $\xi_i = v_i$, $x_n = 0$ and the change of coordinates for $x_n > 0$ is given by $x_j = x_n\xi_j$.\\

The next lemma is the main fact that we need about blow-ups of actions and implies that for a given group action $\Phi: G \to \Dc(N)$, preserving a submanifold $L$, there is a corresponding action $\Phi': G \to \Dc(N^{\sigma})$ in the corresponding blow-up $N^{\sigma}$, more precisely:

\begin{lemma}\label{blow3}

There is a ``blow-up'' map $\sigma: \Dc(N, L) \to \Dc(N^{\sigma})$ such that for every $h \in \Dc(N, L)$ the following diagram commutes:

\[
 \begin{CD}
 N^{\sigma} @>{\sigma(h)}>>  N^{\sigma} \\
 @V{\pi}VV          @VV{\pi}V \\
N @>>{h}> N
 \end{CD}
\]

Even more, $\sigma$ is a group homomorphism.

\end{lemma}

\begin{proof}

For any $h$ in $\Dc(N)$ we will define $\sigma(h)$ as follows: If $x \in N^{\sigma} \setminus \partial N^{\sigma} $ we define $\sigma(h)(x) = h(x)$. On the boundary $ \partial N^{\sigma}$, we can define ${\sigma}(h)$ as the projectivization (up to positive scale) of the map $Dh(p): T_pN \to T_pN$. It is easy to see that $\sigma(h)$ is continuous.\\

We need to check that $\sigma(h)$ is a smooth map: Suppose that in our coordinate charts $x_1,x_2, \dots ,x_n$ around $p$ in $N$, the diffeomorphism $h$ is given locally by $h= (h_1, \dots ,h_n)$, where $p$ correspond to $0 \in \R^n$ and $h(0) = 0$. Given a  direction $v$ (up to positive scale) at the origin, such as $v = \langle v_1,v_2, \dots ,v_{n-1},1 \rangle$, there exists some $j$ such that $\nabla h_j (p)(v_i , 1) \neq 0$, so we can assume that $\nabla h_j (p)(v_i , 1) > 0$. In our coordinate charts  $\xi_1,\xi_2, \dots ,\xi_{n-1},x_n$ explained previously, we have that $\sigma(h)$ is given by  $\sigma(h) = (h_1^{\sigma},h_2^{\sigma},  \dots ,h_n^{\sigma})$ where:

\[
h^{\sigma}_i(\xi, x_n) = 
  \begin{cases}

{\frac{h_i(x_n \xi , x_n)}{h_j(x_n\xi , x_n)}}  & \text{ if } i \neq j\\

 h_j(x_n\xi, x_n)                   & \text{    if }  i = j\\
  \end{cases}
\]
 
Where $\xi := \langle \xi_1,\xi_2, \dots ,\xi_{n-1} \rangle$ and $x_n \neq 0$.\\

When $x_n = 0$ (in the boundary), we have that:
$$h_i^{\sigma}(\xi, 0) = \underset{x_n \to 0}\lim \frac{h_i(x_n\xi, x_n)}{h_j(x_n\xi, x_n)} = \frac {\nabla h_i(0) \langle\xi, 1\rangle}{\nabla h_j(0) \langle \xi, 1\rangle} $$

Additionally, observe that $\frac{\partial}{\partial x_n} h_j(x_n\xi, x_n) = \nabla h_j(x_n\xi, x_n) (\xi, 1) $ and that ${\nabla h_j(0) (\xi, 1)} \neq 0$. \\

All the previous formulas imply that  the derivatives of any order of ${\sigma}(h)$ exist  for points in the boundary $\partial N^{\sigma}$ and therefore ${\sigma}(h)$ belongs to $\Dc(N^{\sigma})$.\\

We still need to show that $\sigma: \Dc(N,L) \to \D(N^{\sigma})$ is a group homomorphism, more precisely,  we need to show that  for any $f,g \in \Dc(N,L) $, the equality ${\sigma}(fg) = {\sigma}(f){\sigma}(g)$ holds. This equality is obvious for points in $N^{\sigma} \backslash \partial N^{\sigma}$ and it follows from the chain rule for points in $\partial N^{\sigma}$.

\end{proof}

For more details and other aspects about the previous  blow-up construction, see (\cite{KM}, sec 2.5). A related blow-up for diffeomorphisms is studied in great detail in \cite{AK}.

\end{subsection}

\begin{subsection}{Gluing actions along a boundary}\label{parkhe}

Given $N_1$ and $N_2$ two compact manifolds with non-empty boundary, suppose there is a diffeomorphism $\alpha: \partial N^1 \to \partial N^2$. Denote $N = N^{1}\cup N^{2}$ the manifold obtained by gluing $N^{1}$ and $N^{2}$ along the boundary using $\alpha$. Suppose we have actions $\Phi_i: G \to \Dc(N^i)$ in such a way that the actions of every element $ g \in G$ coincide in the common boundary. More precisely, suppose that for each $g \in G$, the equality $\Phi_2(g) = \alpha \circ \Phi_1(g) \circ \alpha^{-1} $ holds for points in $\partial N^2$.\\

In general, if we glue $\Phi_1(g)$ and $\Phi_2(g)$ along the common boundary, we will not obtain a $C^{\infty}$ diffeomorphism of $N$. Nonetheless, the following theorem of Parkhe (See \cite{P}) gives a way to modify the actions in such a way that the resulting glued action of $G$ in our manifold $N$ is a truly $C^{\infty}$ action.

\begin{theorem}\label{kiran}

There are homeomorphisms $\Psi_1$ and $\Psi_2$ of $N^{1}$ and $N^{2}$ with the following property: 

\begin{enumerate}

\item For any $f \in \Dc(N^{i})$,  the map $\Psi_i f \Psi_i^{-1} \in \Dc(N^{i})$
\item If $f_1 \in  \Dc(N^{1})$ and $f_2 \in  \Dc(N^{2})$ coincide in the common boundary, then the diffeomorphism $f$ defined by $\Psi_i f_i \Psi_i^{-1}$ in $N^{i}$ for $i=1,2$ is a $C^{\infty}$ diffeomorphism of $N$.

\end{enumerate}

\end{theorem}

Therefore, if we define new homomorphisms $\Phi'_i : G \to \Dc(N_i)$ by  $\Phi'_i(g) = \Psi_i \Phi_i(g) \Psi_i^{-1}$, we can glue both $\Phi'_1$ and $\Phi'_2$ to obtain a homomorphism $\Phi': G \to \Dc(N)$.

\end{subsection}
 
\end{section}

\begin{section}{Proof of theorem \ref{1}}\label{proof}

In this section, we will use the results of section \ref{main} together with the facts discussed in section \ref{blowups} to give a proof of Theorem \ref{1}. This result along with Theorem \ref{cont} completes also the proof of Theorem \ref{2}.\\

From now on, we denote $G = \Dc(\R^{n})$. As in section \ref{conti}, if $B$ is an embedded open ball contained in $\R^n$,  we denote: $$G_B = \{ f \in \Dc(M) | \text{ supp}(f) \subset B \}$$ 

In the proof of Theorem \ref{1} we are going to make use of the following basic fact:

\begin{lemma}\label{weakly} 

$\Phi$ is \emph{weakly continuous} if for some ball $B \subset M$ whose closure is a closed ball embedded in $M$, we have that the  restriction $\Phi|_{G_B}$ of $\Phi$ to the subgroup $G_B$ is continuous.

\end{lemma}

\begin{proof}

Conjugating with elements of $\D_c(M)$ we get that for any embedded ball $B'$, the homomorphism $\Phi|_{\Dc(B')}$ is continuous. From the proof of the fragmentation lemma (Lemma 2.1.8 in \cite{B}) , one obtains that if the sequence $f_n$ tending to $\text{Id}$ is supported in a compact set $K$ and $\mathcal{C}$ is a finite covering of $K$ by  sufficiently small balls, then for large $n$, the diffeomorphisms $f_n$ can be written as  products of diffeomorphisms supported in open sets of $\mathcal{C}$. Furthermore, we can achieve this with at most $k$ factors for each open set of $\mathcal{C}$ (for some constant $k$) and it can be done in  such a way  that these diffeomorphisms are also converging to the identity.
\end{proof}

We will proceed to give a sketch of our main result, Theorem \ref{1}.

\begin{proof}[Sketch of the proof of Theorem \ref{1}]

The proof  goes by induction in the dimension of $N$. Consider any open ball $B$ whose closure is a closed ball embedded in $M$. By Lemma \ref{weakly}, to show that $\Phi$ is weakly continuous, it is enough to show that $\Phi$ is weakly continuous when we restrict to the subgroup $G_B$.\\ 

We are going to use many times the following fact:  $\Phi$ is weakly continuous if the following condition holds:  There is no sequence $f_n$ in $G_B$, such that $f_n \to \text{Id}$ and $\Phi(f_n) \to A$, where $A$ is a non-trivial isometry in some Riemannian metric on $N$. This fact is an immediate consequence of Lemma \ref{m2}.\\

 Assume such non-trivial isometry $A$ exists. If we take a ball $B'$ disjoint from $B$, the actions of  $G_{B'}$ and $A$  commute. If $A$ has a non-trivial fixed point set $L$, then $G_{B'}$ preserves $L$ setwise, a manifold of lower dimension than $N$. Therefore, we can blow-up the action at $L$ and show by induction that $\Phi$ is continuous.\\ 
 
 If $A$ acts freely instead, we will obtain an action of $G_{B'}$ on the space $N' :=N/H$, where $H = \overline{\langle A \rangle}$ is the closure of the subgroup generated by $A$. If the group $H$ is infinite, $N'$ is a manifold of lower dimension than $N$ and we can proceed to prove the continuity by induction. If $H$ is finite, we will  show a way to replace $H$ by a similar infinite group of isometries $\hat{H}$ to proceed as in the case where $H$ is infinite.\\
 
 \end{proof}

In the proof of Theorem \ref{1}, we are going to make use of the following known fact:

\begin{lemma}\label{sublie} If $N$ is a closed Riemannian manifold and $H$ is a non-trivial closed connected subgroup of the group of isometries of $N$, then $\text{fix} (H) = \{p \in N \mid h(p) = p \text{ for every } h \in H \}$ is a closed submanifold of positive codimension.

\end{lemma}

We will also need the following technical lemma:\\

\begin{lemma}\label{tec}
 If $\Phi: \D_c(B_0) \to \D_c(N') $ is a weakly continuous homomorphism, then there exists an embedded open ball $B_2 \subset B_1$ and a point $p \in N'$ such that for any diffeomorphism $f$ supported in $B_2$, the equality $\Phi(f)p = p$ holds.

\end{lemma}

\begin{proof}

Define $ n = \dim(N)+1$ and consider an arbitrary point $p \in N$. For $1\leq i \leq n$,  let $U_{i}$ be disjoint embedded balls contained in $B_1$. Using Lemma \ref{inj1}, we conclude that for at least one of these balls, which we denote by $B_2$, the restricted homomorphism $\Phi |_{G_{B_2}}$ fixes $p$. 

\end{proof}

By Lemma \ref{weakly}, Theorem \ref{1} is a direct consequence of the following:

\begin{lemma}\label{3}

Let $N$ be a closed manifold. If $\Phi: {\Dc}(\R^m) \to {\Dc}(N)$ is a group homomorphism, then $\Phi$ is weakly continuous.

\end{lemma}

\begin{proof}

We proceed by induction on the dimension of $N$. The proof for the case $n =0$ is trivial and the case $n=1$ is proved in \cite{M}.\\

Let $B$ de the unit ball in $\R^n$. Suppose that  $\Phi$ is not weakly continuous. By Lemma \ref{m2},  there is a ball $B_0  \subset B$ strictly contained in $B$ and a sequence $f_n \in G_{B_0}$ such that $f_n \to \text{Id}$ and $\Phi(f_n) \to A$, where $A$ is a non-trivial isometry for some Riemmanian metric on $N$. Define $H := \overline{\langle A \rangle}$, $H$ is a closed subgroup of the Lie group of isometries of $N$. We can further suppose that $B_0$ is strictly contained in $B$.

For the rest of the proof, we let $B_1 \subseteq B$ be an embedded ball disjoint from $B_0$, Observe that as  $B_0$ and $B_1$ are disjoint, the groups $\Phi(G_{B_0})$ and $\Phi(G_{B_1})$ commute.

First we show that we can assume $H$ acts freely:\\

\begin{lemma}\label{part1} If $H$ does not act freely, then $\Phi$ is weakly continuous.
\end{lemma}

\begin{proof}

Suppose that $H$ does not act freely. We can replace our group $H$ with a non-trivial closed subgroup and assume that $L = \text{fix}(H) \neq \emptyset$. By Lemma \ref{sublie}, $L$ is a closed submanifold of positive codimension.\\  

Observe that the action of $\Phi(G_{B_1})$ on $N$ commutes with the action of $H$ and therefore $L$ is invariant under $\Phi(G_{B_1})$. Using Lemma \ref{blow3} we can blow-up the action $\Phi$ at $L$ to obtain an action $\Phi^{\sigma}: G_{B_1} \to N^{\sigma}$. This new action preserves the boundary $\partial N^{\sigma} = N(S)$, and therefore we obtained a homomorphism $\Phi^{\sigma}|_{\partial N^{\sigma}}: {\Dc}(\R^m) \to {\D}(\partial N^{\sigma})$ that by induction is weakly continuous.\\  

To prove that $\Phi$ is weakly continuous, we will need to define the following objects: Let $N^1$ and $N^2$ be two copies of $N^{\sigma}$. Additionally, let $N' := N^{1} \cup N^{2}$, be the manifold obtained by gluing $N^{1}, N^{2}$ in the obvious way. Finally, consider the homomorphisms $\Phi_i: G_{B_1} \to \Dc(N^{i})$ equal to $\Phi^{\sigma}$ for $i =1,2$.\\

Observe that the previous homomorphisms define an action of $G_{B_1}$ on $N'$, but this action might not be smooth in the submanifold corresponding to the glued boundaries. Using the gluing theorem described in section \ref{parkhe}, we can conjugate the actions in each $N^i$ by homeomorphisms and obtain a smooth action $\Phi': G_{B_1} \to \Dc(N')$.\\

We are now going to make use of the following trick: We will show that $\Phi'$ is weakly continuous using the induction hypothesis and then we will show that  the continuity of  $\Phi'$ implies the continuity of $\Phi$.\\

Let's show that $\Phi': G_{B_1} \to \Dc(N')$ is weakly continuous. By Lemma \ref{m2}, it is enough to show the following: If $f_n \in G_{B_1}$ is a sequence such that  $f_n \to \text{Id}$ and $\Phi'(f_n) \to A$, where $A$ is an isometry for $N'$, then $A$ is trivial. 

By induction, the restriction of $\Phi'$ to the submanifold corresponding to the glued boundaries is continuous,  therefore $A$ is trivial in a codimension one submanifold. As any orientation-preserving isometry that is trivial in a codimension one submanifold is trivial everywhere,  we obtain that $A$ must be trivial.\\

Now we show that the weak continuity of  $\Phi'$ implies  $\Phi$ is weakly continuous as well. Let $f_n$ be a sequence in $G_{B_1}$ such that $f_n \to \text{Id}$ and $\Phi(f_n) \to A$, where $A$ is an isometry of $N$.  From the previous paragraph, we have that $\Phi'(f_n) \to \text{Id}$. This implies that $\Phi^{\sigma}(f_n) \to \text{Id}$ (at least topologically) and therefore $\Phi(f_n) \to \text{Id}$. In conclusion,  $A$ must be trivial and by Lemma \ref{m2}, the homomorphism $\Phi$ is continuous.

\end{proof}

We can now assume that the action of $H$ in $N$ is free. Next, we will prove that if $H$ is infinite, then $\Phi$ is continuous.

\begin{lemma}\label{part2} If $H$ is infinite and acts freely, then $\Phi$ is weakly  continuous.

\end{lemma}

\begin{proof}

Let $B_1$ a ball in $M$ disjoint from $B_0$. The actions of $\Phi(G_{B_1})$ and $H$ on the manifold $N$ commute. If we let $N' := N/H$, then  $N'$ is a smooth manifold because the action is free. The dimension of $N'$ is lower than the dimension of $N$, moreover, we have that  $\Phi$ descends to a homomorphism $\Phi':  G_{B_1} \to \Dc(N')$.\\

By induction, $\Phi'$ is continuous. Let $p \in N'$ be an arbitrary point.  By Lemma \ref{tec}, there is a ball $B_2 \subseteq B_1$ such that $\Phi'(G_{B_2})p = p$. Therefore, $\Phi(G_{B_2})$ preserves the manifold $H_p = \{ h(p) \mid h \in H \}$, furthermore we can assume that the dimension of $H_p$ is lower than the dimension of $N$. In conclusion, we obtain a homomorphism $\Phi'': G_{B_2} \to \Dc(H_p)$ that by induction is continuous.\\

We are going to deduce the weak continuity of $\Phi$ from the continuity of $\Phi'$ and $\Phi''$. By Lemma \ref{m2}, it is enough to show that if $f_n \in G_{B_2}$ is a sequence such that $f_n \to \text{Id}$ and $\Phi(f_n) \to A$, where $A$ is an isometry of $N$, then $A = \text{Id}$.\\

Take a point $q \in H_p$, we are going to show that  $A(q) = q$ and that the derivative $D_q(A)$ of $A$ at $q$ is trivial. Given that $A$ is an isometry of some metric, these two facts are enough to show that $A$ is trivial. The fact that $A(q) = q$, follows from the fact that $\Phi''(f_n) \to \text{Id}$.\\

We now prove that the derivative $D_qA$ is trivial. We can split the tangent space $T_q(N)$ of $N$ at $q$ as follows: $$T_q(N) = T_q(H_p) \oplus W$$ Where $T_q(H_p)$ is the tangent space of $H_p$ at $q$ and $W$ is the orthogonal complement of $T_q(H_p)$ in $T_q(N)$. Observe that as  $\Phi''(f_n) \to \text{Id}$,  then $D_qA$ preserves $T_q(H_p)$ and $D_q|_{T_q(H_p)} = \text{Id}$. In a similar way, considering the fact that $\Phi'(f_n) \to \text{Id}$,  if we let $\pi_W$ denote the orthogonal projection of $T_q(N)$ onto $W$, we have that $\pi_{W}\circ D_qA = \text{Id}$\\

These two facts imply that the linear map $D_q(A$) must be the identity: $D_q(A)$ has a matrix decomposition in blocks, each block corresponding to the subspaces  $T_q(H_p)$ and $W$. In this decomposition, $D_qA$ is an upper triangular matrix and the diagonal blocks are the identity, therefore,  all the eigenvalues of $D_qA$ are equal to one. Since $A$ is an isometry, $D_qA$  is an orthogonal matrix will all its eigenvalues equal to one and therefore $D_qA$ must be the identity.\\ 

In conclusion, we have that $A(q)=q$ and $D_qA = \text{Id}$. As $A$ is an isometry, $A$ must be trivial.\\

\end{proof}

The only remaining case to finish the proof of Theorem \ref{3} is the case where $H$ is a finite group and $H$ acts freely on $N$. We will show that if $\Phi$ were not continuous, we can replace $H$ with an infinite group of isometries $\hat{H}$  that would reduce the proof of this case to the previous two lemmas.\\

\begin{lemma}\label{part3} If $H$ is finite and acts freely, one of the following alternatives hold:

\begin{enumerate} 

\item$\Phi$ is weakly continuous.
\item There exists a metric on $N$ and an infinite  subgroup of isometries $\hat{H}$ of $N$, together with an embedded ball $B' \subseteq \R^{n}$ such that the action of $\Phi(G_{B'})$ commutes with the action of $\hat{H}$.

\end{enumerate}

\end{lemma}

\begin{proof}

Let $N_0:= N$, $H_0:= H$ and $N_1 := N_0/H_0$. The homomorphism $\Phi$ descends to a homomorphism $\Phi_1: G_{B_1} \to \D(N_1)$ as the groups $G_{B_1}$ and $H_0$ commute with each other. One can check that if $\Phi_1$ is continuous, then $\Phi$ is continuous. Therefore, if $\Phi_1$ is not  continuous, there is a sequence $f_{n,1} \in G_{B_1}$ such that $f_{n,1} \to \text{Id}$ and such that $\Phi_1(f_{n,1}) \to A_1  \in \D(N_1)$, where $A_1$ is a non-trivial isometry on some metric on $N_1$.  Let $H_1 = \overline{\langle A_1\rangle}$. If $H_1$ does not act freely  or $H_1$ is infinite, we can use lemmas \ref{part1} and \ref{part2} to conclude that $\Phi_1$ is continuous, therefore we can assume $H_1$ is finite and fixed point free.\\ 

We are now in the same situation we were at the beginning of the proof of this lemma. For that reason, we can repeat this procedure infinitely many times as follows:\\

We construct  for each integer $k$, the following objects: A ball $B_k \subset B$ disjoint from the balls $B_j$ for $j < k$, a closed manifold $N_k = N_{k-1}/H_{k-1}$ and a group homomorphism $\Phi_k: G_{B_{k}} \to {\D}(N_{k})$ descending from $\Phi_{k-1}$. Together with a sequence $\{f_{n,k}\} \in G_{B_k} $ such that as $n\to \infty$, we have that $f_{n,k} \to \text{Id}$ and  $\Phi_k(f_{n,k}) \to A_k$, where $A_k$ is a non-trivial isometry on some metric on $N_k$ and the group $H_k = {\langle A_k\rangle}$ is a non-trivial finite group of diffeomorphisms acting freely on $N_{k}$. \\

Observe that the only way this procedure can't be repeated infinitely times is if for some $k$, the homomorphism  $\Phi_k$ were weakly continuous, but this would imply that $\Phi$ is weakly continuous, therefore we can assume the previous objects are defined for every integer $k \geq 1$. Next, we are going to show how to get an infinite group of isometries for $N$ from our $H_k$'s.  Let's define: $$\hat{H_k} = \{ h \in {\Dc}(N) \mid h \text{ is a lift of some diffeomorphism } h_k \in H_k \}$$ 

Each diffeomorphism $h_k$ of $H_k$ can be lifted to a diffeomorphism of $N$:  The diffeomorphism $A_k$ is the limit of the sequence  $\Phi_k(g_{n,k})$, we consider the sequence $\{\Phi(g_{n,k})\}$ that consist of lifts of $\Phi_k(g_{n,k})$, by Lemma \ref{m2} this sequence has a convergent subsequence converging to a diffeomorphism $A'_k$ of $N$ which is necessarily a lift of $A_k$.\\

Observe that in fact $\hat{H_k}$ is a finite subgroup of  ${\Dc}(N)$ and that $\hat{H}_{k-1} \subset \hat{H_k}$, being this last inclusion strict.\\

We define the group $\hat{H} := \bigcup_{k\geq0} \hat{H_k}$. Observe that $\hat{H}$ is necessarily an infinite group. We will prove that $\hat{H}$ is invariant under a Riemannian metric on $N$ in a similar way that we prove Lemma \ref{m2}. In order to do that, we will show the following bounds:\\

\begin{lemma}\label{bounds}
Every  $h \in \hat{H}$ satisfies the following bounds:

\begin{enumerate}
\item $\|h\|_r \leq C_r$ for every $r \geq 1$. 
\item $\|D(h)\| \leq C$.
\end{enumerate}
 For the constants $C, C_r$ in Lemma \ref{m1} 
\end{lemma}

\begin{proof}

We will prove by induction in $k$ that for every $h \in \hat{H_k}$, there exists a sequence $h_n \in G_B$ such that $h_n \to \text{Id}$ and $\Phi(h_n) \to h$. By lemma \ref{m1}, this is enough to show that such bounds in the derivatives hold.\\

This statement is obvious for $k= 0$. Suppose it is true for $j \leq k-1$. Consider the projection homomorphism  $\Psi: \hat{H_k} \to H_k $ and observe that the kernel of $\Psi$ is exactly $\hat{H}_{k-1}$.\\

By definition $H_k = \langle A_k \rangle$. First, we will show that $A_k$ has a lift to $N$ belonging to $\hat{H_k}$ and satisfying our induction hypothesis. Consider the sequence $f_{n,k} \in G_{B_k} $ such that as $n \to \infty$, we have that $f_{n,k} \to \text{Id}$ and  $\Phi_k(f_{n,k}) \to A_k$  as in the construction of the $A_k$'s. We have that for every $n$, the diffeomorphism $\Phi(f_{n,k})$ is a lift of $\Phi_k(f_{n,k})$ to $\Dc(N)$. Using Lemma \ref{m2}, we can find a subsequence  $f_{{n_i},k}$ of $f_{{n},k}$, such that $\Phi(f_{{n_i},k}) \to A_k'$, where $A_k'$ is a diffeomorphism of $N$.  The diffeomorphism  $A_k'$ is necessarily a lift of $A_k$ and therefore $A_k'$ belongs to $\hat{H_k}$ and satisfies our induction hypothesis. \\

To finish the induction, observe that as $H_k$ is cyclic, every element of $\hat{H_k}$ is a product of an element of $\text{ker}(\Psi) = \hat{H}_{k-1}$ and a power of $A_k'$. Therefore, to conclude that every element of $\hat{H_k}$ satisfies the induction hypothesis, it is enough to show that if $f,g \in \hat{H_k}$ satisfy the induction hypothesis, then $fg$ satisfies the hypothesis too. This is easy to prove, if there  exist sequences $\{f_n\}, \{g_n\}$ such that $f_n \to \text{Id}$, $g_n \to \text{Id}$ and such that $\Phi(f_n) \to f$, $\Phi(g_n) \to g$, then the sequence $\{f_ng_n\}$ satisfies that $f_ng_n \to \text{Id}$  and $\Phi(f_ng_n) \to fg$.

\end{proof}

To conclude the proof of Lemma \ref{part3}, we would proceed in a similar way as we did in Lemma \ref{m2}, we will show that $\hat{H}$ preserves a Riemannian metric $g$ in $N$ as follows:  Take an arbitrary metric $g'$ for $N$ and average $g$ with respect to the groups $\hat{H_n}$. More precisely, consider the sequence of metrics:
 
 $$g_n = \frac{1}{|\hat{H_k}|} \underset{h \in \hat{H_k}}\sum h^{*}(g')$$
 
Observe that $g_n$ is invariant under $\hat{H_n}$.  Therefore, if there is a  subconvergent sequence of $g_n$ converging to a metric $g$, then $g$ is invariant under $\hat{H}$. To ensure the existence of such subconvergent subsequence using the same argument used in the proof of Lemma \ref{m2} taking into account the bounds obtained in Lemma \ref{bounds}. In conclusion, $\hat{H}$ is an infinite group of isometries of $N$. To finish the proof of Lemma \ref{part3},  take an embedded ball $B' \subset B$ in $N$ disjoint from all the balls $B_n$ and apply Lemma \ref{part2}.

\end{proof}

\end{proof}

\end{section}

 \begin{section} {Questions and Remarks.}\label{remarks}
 
 

 \begin{subsection}{Higher dimensions}

Maybe the most natural question to ask in view of Theorem \ref{2} is wether it is possible to obtain a characterization of homomorphisms $\Phi: \Dc(M) \to \Dc(N)$ in the case that $\dim(M) < \dim(N)$. All the known homomorphisms known to the author are built from pieces,  each piece coming from some natural bundles over $M$ or over $Symm_k(M)$ (the set of unordered $k$ points in $M$). These bundles could be products, coverings, or somewhat similar to the tangent bundle or other bundles where one has some kind of linear action in the fiber. The pieces are glued along submanifolds where the actions in both sides agree.

 \end{subsection}

\begin{subsection}{Lower regularity}

One might expect similar results for the groups $\Dc^{r}(M)$ of $C^{r}$ diffeomorphisms of a manifold $M$. Some of the techniques here might apply to such problems in the case when $r \geq 2$ or $r \geq 3$. One of the reasons to believe it is possible to do so is that $\Dc^{\infty}(M) \subset \Dc^{r}(M)$ and  therefore one might be able to show some sort of continuity for homomorphisms of the type $\Phi: \Dc^{\infty}(M) \to \D^{r}(M)$ using Militon's theorem.\\ 

It is important to point out  that Militon's theorem might be very difficult to generalize to the $C^{r}$ category: One of the main steps in the proof of Militon's Theorem uses a KAM technique and such techniques typically have a loss of regularity if $r < \infty$. In the topological category things seem to be much more difficult as every homeomorphim $f \in \text{Homeo}_0(M)$ is \emph{arbitrarily distorted} (See \cite{CF}).
 
 \end{subsection}

 \begin{subsection}{Discrete homomorphisms between Lie groups}
 
One might also try to understand discrete homomorphisms between simple Lie groups. For example, one can show that any discrete homomorphism from $SO(3)$ to itself is conjugate to the standard one. The same is true for $SL_2(\R)$. For $SL_2(\C)$, the situation is a little bit different,  one can use a non trivial field automorphism of $\C$, to get a homomorphism from $Sl_2(\C)$ to itself that is not conjugate to the standard one. Nonetheless, one can show that any of the homomorphism between $Sl_2(\C)$ and itself come from this construction.\\

An observation worth mentioning is that if the homomorphism $\Phi$ is measurable, then $\Phi$ is continuous and in fact $C^{\infty}$ (See \cite{Z}, Ap. B.3). Therefore, a non standard homomorphism is necessarily not measurable (as in the examples for $SL_2(\C)$ described above).\\
 
Other more difficult set of questions related to the Zimmer program (see \cite{FIS}) come from asking wether a discrete homomorphism of groups from a simple lie group $G$ to $\D_c(M)$ comes from a standard embedding. For example, Matsumoto\cite{MAT} shows that any action of $PSL_2(\R)$ in $\Dc(S^1)$ is conjugate to the standard action. One might wonder if that is also the case for $S^2$ and $PSL_3(\R)$: \\

 \begin{subsubsection}{Question} Does every homomorphism $ \Phi: PSL_3(\R) \to \D(S^2)$ is conjugate to the standard embedding?\\
 \end{subsubsection}

Using the fact that $PSL_2(\C)$ acts in $S^{2}$  and $SO(3) \subset PSL_2(\C)$ one can compose the action of $ PSL_2(\C)$ by a non trivial field automorphism of $\C$ to obtain a non standard action of $SO(3)$ in $S^2$, nonetheless one can still ask if the following is true in the volume preserving setting:

 \begin{subsubsection}{Question} Does there exist a homomorphism $\Phi: SO(3) \to {\D_{\mu}}(S^2)$ that is not conjugate to the standard one?\\
  \end{subsubsection} 
 
For this type of questions, the technique used here does not seem to work as the distortion elements in any linear group are at most exponentially distorted. Nevertheless, studying the distorted elements involved might give some useful information.

 \end{subsection}
 
  \begin{subsection}{Distortion elements in groups of diffeomorphisms}
  
 It is worth pointing out that Militon's theorem implies that if $M$ is a closed manifold and $f$ is an infinite order isometry on $M$ then $f$ is arbitrarily distorted. The converse to this statement is not true, there are examples of diffeomorphisms $f$ in $\Dc(M)$ that are not isometries but  are still recurrent ($f$ is recurrent if satisfies $\liminf_n d_{C^{\infty}}(f^n,\text{Id}) = 0$) and therefore arbitrarily distorted as a consequence of Militon's theorem. Some of these examples can be constructed using the Anosov-Katok method. For example, one can use the construction in \cite{Her} to obtain such examples.
 
 \begin{subsubsection}{Question} Is it possible to obtain a classification of all possible distorted (or arbitrarily distorted as defined in \cite{CF}) elements in $\Dc(S^1)$ or $\Dc(S^2)$?\\
 
 \end{subsubsection}
 
This might be useful to show that certain discrete groups can't act by diffeomorphisms on $S^1$ and $S^2$.

 \end{subsection}

 \end{section}

\end{document}